\documentclass{amsart}
\usepackage[utf8]{inputenc}

\usepackage[margin=1.125in]{geometry}

\usepackage{xcolor}

\usepackage{tikz}
\tikzset{graph-vertex/.style={circle,minimum width=7pt,draw,inner sep=1pt}}

\usepackage{kbordermatrix}

\usepackage{amsmath}
\usepackage{amsthm}
\usepackage{amsfonts}
\usepackage{mathtools}
\usepackage{enumitem}
\setlist[enumerate,1]{label = (\arabic*)}

\usepackage[reftex]{theoremref}
\usepackage{hyperref}

\hypersetup{
  colorlinks   = true, %
  urlcolor     = blue, %
  linkcolor    = blue, %
  citecolor   = blue %
}
\usepackage{graphicx}
\graphicspath{ {./images/} }
\usepackage[mathlines]{lineno}

\newtheorem{thm}{Theorem}[section]

\newtheorem{lem}[thm]{Lemma}
\newtheorem{cor}[thm]{Corollary}
\newtheorem{prop}[thm]{Proposition}
\newtheorem{obs}[thm]{Observation}

\theoremstyle{definition}
\newtheorem{q}[thm]{Question}
\newtheorem{defn}[thm]{Definition}
\newtheorem{ex}[thm]{Example}

\numberwithin{equation}{section}

\newcommand{\Y}{\mathcal{Y}}
\DeclareMathOperator{\mr}{mr}
\DeclareMathOperator{\tri}{tri}

\DeclareMathOperator{\tw}{tw}

\newcommand{\gqual}[1]{\mathcal{S}(#1)}
\newcommand{\Sznz}{\mathcal{S}_{\mathrm{znz}}}

\newcommand{\defterm}[1]{\emph{#1}}

\title{Relationships between minimum rank problem parameters for cobipartite graphs}

\author{Louis Deaett}
\address{Department of Mathematics and Statistics, Quinnipiac University, Hamden, CT 06518, USA (louis.deaett@quinnipiac.edu)}

\author{Derek Young}
\address{Department of Mathematics and Statistics, Mount Holyoke College, South Hadley, MA 01075, USA (dyoung@mtholyoke.edu)}

\keywords{minimum rank; inverse eigenvalue problem; matrix pattern; zero forcing; maximum nullity}

\subjclass[2020]{Primary: 05C50; Secondary: 15B35}

\begin{document}

\begin{abstract}
For a simple graph, the minimum rank problem is to determine the smallest rank among the symmetric matrices whose off-diagonal nonzero entries occur in positions corresponding to the edges of the graph. Bounds on this minimum rank (and on an equivalent value, the maximum nullity) are given by various graph parameters, most notably the zero forcing number and its variants. For a matrix, replacing each nonzero entry with the symbol \(\ast\) gives its zero-nonzero pattern.  The associated minimum rank problem is to determine, given only this pattern, the smallest possible rank of the matrix. The most fundamental lower bound on this minimum rank is the triangle number of the pattern. A cobipartite graph is the complement of a bipartite graph; its vertices can be partitioned into two cliques. Such a graph corresponds to a zero-nonzero pattern in a natural way.  Over an infinite field, the maximum nullity of the graph and the minimum rank of the pattern obey a simple relationship.  We show that the zero forcing number of the graph and the triangle number of the pattern follow this same relationship. This has implications for the relationship between the two minimum rank problems.  We also explore how, for cobipartite graphs, variants of the zero forcing number and other parameters important to the minimum rank problem are related, as well as how, for graphs in general, these parameters can be interpreted in terms of the zero-nonzero patterns of the symmetric matrices associated with the graph.
\end{abstract}

\maketitle

\section{Introduction}\label{sec:intro}

The question of how a combinatorial description of a matrix allows bounds to be
placed on its rank is a well-studied one (see, e.g., \cite{not_necessarily_symmetric,fgggjlsz2018,hs93,johnson_and_link,johnson_and_zhang}).
Our work here concerns two different variants of this problem, one for zero-nonzero patterns, and one for simple graphs.

The \defterm{zero-nonzero pattern} of a matrix is a specification of exactly which entries of the matrix are zero and which are nonzero.
More formally, we say that a \defterm{zero-nonzero pattern matrix} is a matrix with entries from the set $\{0,*\}$.
Where no confusion may result, we refer to a zero-nonzero pattern matrix as just a \defterm{zero-nonzero pattern} or even simply as just a \defterm{pattern}, when the context makes it clear what is meant.

Given a matrix $A$, the \defterm{zero-nonzero pattern of} $A$ is the matrix that results from replacing each nonzero entry in $A$ with a $*$.  
Given a zero-nonzero pattern matrix $\Y$, the \defterm{minimum rank} of $\Y$ over the field $\mathbb F$ is the smallest rank of a matrix with entries in $\mathbb F$ and the zero-nonzero pattern $\Y$.  
We write $\mr(\Y)$ to denote the minimum rank of $\Y$ over the real numbers or some other field given in context.

The problem of relating the combinatorics of a pattern to its minimum rank is known as the minimum rank problem for zero-nonzero patterns.  For this problem, a fundamental lower bound is provided by the \defterm{triangle number} of the pattern, which we now define.

\begin{defn}\label{def:triangle submatrix and triangle number}
A square zero-nonzero pattern is said to be a \defterm{triangle}
if some permutation of its rows, followed by some (independent) permutation of its
columns results in a pattern that is lower-triangular with only $\ast$ entries
on its diagonal. The \emph{triangle number} of a pattern \( \Y \),
denoted by \( \tri(\Y) \), is
the size of a largest submatrix of \( \Y \) that is a triangle.
\end{defn}

A natural question then becomes:\ Under what conditions is the minimum rank of a pattern actually given by its triangle number?  For example, this question was the focus of \cite{johnson_and_canto, johnson_and_link,  johnson_and_zhang}.  In examining small examples, equality between the minimum rank and the triangle number seems common, and the phenomenon whereby a gap may exist between the two remains poorly understood.

A different minimum rank problem,
the minimum rank problem for graphs,
starts with a simple graph $G$ and asks for the smallest rank of a symmetric matrix whose zero-nonzero pattern off the diagonal matches that of the adjacency matrix of $G$, a value called the \defterm{minimum rank} of $G$.  (See \cite[Section 2.1]{hls22} and the survey \cite{fh07}.)  In terms of the
inverse eigenvalue problem for graphs \cite{hogben2005}, this is equivalent to finding the maximum multiplicity of an eigenvalue of a matrix meeting that combinatorial description, a value equal to the maximum nullity of such a matrix.
This is called the \defterm{maximum nullity} of $G$.

Connecting the two minimum rank problems described above is the chief motivation for the present paper, which proceeds as follows.
Section \ref{sec:preliminaries} establishes definitions and terminology needed throughout the paper, and recalls some inequalities relating various combinatorial parameters of a graph that play a role in bounding its minimum rank.
We also introduce the class of cobipartite graphs, those with a complement that is bipartite.

Section \ref{sec:triangles-and-forcing-sets} explores how
triangles in the zero-nonzero patterns of the symmetric matrices associated with a graph correspond to zero forcing sets for the graph that obey certain constraints.  A number of zero forcing parameters are ones we can interpret in terms of the largest size of a triangle in such a pattern that is in some way independent of the diagonal entries of the pattern.

Section \ref{sec:triangles and forcing of partitioned graphs} examines graphs in which the edges that cross some partition correspond to the $\ast$ entries in a given pattern.  We find that triangles within the pattern correspond to zero forcing sets from which the forcing process can proceed in a way that respects the partition.

Starting with Section \ref{sec:cobipartite}, the focus is on graphs in which each side of the partition induces a clique, graphs that are cobipartite.  
Here we introduce an important nondegeneracy condition, that of the graph being \defterm{saturated}.
\thref{thm:triangle number versus ZF for cobipartite} shows that when a cobipartite graph satisfies this condition, the triangle number of the corresponding pattern and the zero forcing number of the graph are essentially equivalent; each can be expressed as the number of vertices minus the other.  We use this to show that, over an infinite field, whether the zero forcing number of such a graph is equal to its maximum nullity is equivalent to whether the minimum rank of the pattern is equal to its triangle number.  This yields a new example of a $15$-vertex graph (the smallest size known) whose maximum nullity differs from its zero forcing number over every field.
Also in Section \ref{sec:cobipartite}, we explore how zero forcing parameters beyond the standard zero forcing number behave for cobipartite graphs, and find that the standard zero forcing number, the enhanced zero forcing number, and the positive semidefinite zero forcing number are all equal.  The final part of Section \ref{sec:cobipartite} examines the treewidth parameter for cobipartite graphs.  We show that this need not coincide with the zero forcing parameters just mentioned, but that for the special case of cobipartite $k$-trees, all of them do in fact coincide, and equal the maximum nullity of the graph.

Finally, in Section \ref{sec:unsaturated}, we consider what remains true when a cobipartite graph is not saturated, meaning that some vertices lack a neighbor on the other side of the partition.
There, the main result is \thref{thm:remove_unsaturated_vertices_effect_on_Z}, which shows exactly how such vertices affect the zero forcing number.
On the other hand, the precise effect of these vertices on the maximum nullity seems harder to determine, and we offer this as an enticing open question.
In any case, our results show that the equivalence given by \thref{thm:triangle number versus ZF for cobipartite} holds if and only if the cobipartite graph is in fact saturated.

\section{Preliminaries}\label{sec:preliminaries}
A \emph{loop graph} is a graph that allows loops (but not multiple edges).
Throughout what follows, every graph is a simple graph (i.e., without loops or multiple edges) except when specifically indicated otherwise.  In particular, we often let $\hat G$ be a \defterm{looping} of some specific simple graph $G$, meaning that $\hat G$ is a loop graph produced from $G$ by adding loops to some of the vertices of $G$.
For a graph $G$, we write $|G|$ to denote the number of vertices in the graph.

The \emph{set of symmetric matrices over $\mathbb F$ described by the graph \( G \)}, denoted by
\( \mathcal{S}(G,\mathbb F) \), is the set of symmetric matrices over $\mathbb F$ whose \( (i,j) \)-entry is nonzero if \(
\{ i,j\} \) is an edge in \( G \) and \( 0 \) otherwise, when \( i \neq j \).
When we write only $\gqual{G}$,
we mean to denote $\mathcal S(G, \mathbb F)$ where $\mathbb F$ is the field of real numbers, or some other field given in context.
The \emph{minimum rank of a graph \( G \) over the field \( \mathbb{F} \)} is the minimum rank over \(
\mathcal{S}(G,\mathbb F)  \). We use \(
\mr(G) \) for the minimum rank over the real numbers
or some other field given in context. Similarly, the \emph{maximum nullity of a graph \(
G \) over the field \( \mathbb{F} \)} is the maximum nullity over \(
\mathcal{S}(G,\mathbb F)  \). We use \(
M(G) \) to represent the maximum nullity over the real numbers
or some other field given in context. The \emph{positive semidefinite minimum rank of a
graph \( G \)}, denoted by \( \mr_{+}(G) \), is the minimum rank
over all positive semidefinite matrices of \( \mathcal{S}(G,\mathbb R) \). Similarly, the
\emph{positive semidefinite maximum nullity of a graph \( G \)}, denoted by \(
M_{+}(G) \), is the maximum nullity over all positive semidefinite
matrices of \( \mathcal{S}(G,\mathbb R) \).

Let \( G \) be a simple graph such that every vertex is labeled as filled or
unfilled. The \emph{zero forcing rule} for \( G \) is defined as follows:
a filled vertex \( v \) forces an unfilled vertex \( u \) to be filled if \(
u \) is the only neighbor of \( v \) which is unfilled. 
A set \( F \subseteq V(G) \) is a \emph{zero forcing set} of $G$ if \( V(G) \) is
completely filled after some number of repeated applications of the zero forcing rule. The
\emph{zero forcing number} of a graph \( G \), denoted by \( Z(G) \), is the smallest cardinality of a zero forcing set for \( G \).

The \defterm{positive semidefinite zero forcing rule} for $G$ is defined as follows: considering the subgraph of $G$ formed by deleting all of its filled vertices, if a filled vertex $v$ has just one neighbor $u$ in some connected component of that subgraph, then $v$ can force $u$.
A set \( F \subseteq V(G) \) is a \emph{positive semidefinite zero forcing set} if \( V(G) \) is
completely filled after
some number of repeated applications of the positive semidefinite zero forcing rule.
The
\emph{positive semidefinite zero forcing number} of a graph \( G \), denoted by \( Z_+(G) \), is the smallest cardinality of a positive semidefinite zero forcing set for \( G \).

The \defterm{loop zero forcing rule} for a loop graph is the same as the ordinary zero forcing rule, but without the requirement that a vertex needs to be filled in order to force.
A set of vertices is a \emph{loop zero forcing set} if every vertex in the loop graph becomes filled after
some number of repeated applications of the loop zero forcing rule.
By the \defterm{zero forcing number} of a loop graph we mean the smallest size of a loop zero forcing set for that loop graph.

A separate zero forcing parameter that is defined for a simple graph, but in terms of loopings of that graph, is the \defterm{enhanced zero forcing number}, denoted $\hat Z(G)$, which is the maximum value of the zero forcing number of a looping $\hat G$ of $G$.

We now define some notation and terms for all zero forcing rules. In the case where
\( v \) forces \( u \) to become filled, we say \( v \) \defterm{forces} \( u \), 
or write \(
v \to u \), 
and refer to this as a single \defterm{force} (or \defterm{standard force}).
We note that any force that could occur under the zero forcing rule could also occur under the positive semidefinite zero forcing rule.  Of course, the converse is not true.  When we wish to indicate that a particular force could occur under the positive semidefinite zero forcing rule but not under the ordinary zero forcing rule, we refer to the force as a \defterm{strictly positive semidefinite force}.

The following notions, originally introduced for simple graphs in \cite{ZF_parameters_and_min_rank}, are used here with slightly different terminology.
A \emph{forcing sequence} is a (possibly empty) sequence of
forces.
A \emph{forcing sequence from \( F \subseteq V(G) \)} (with respect to some fixed forcing rule) is a sequence of forces, say of length $k$, such that, starting with $F$ as the initially filled set, for all $i$ with $1 \le i \le k$, after the first $i-1$ forces in the sequence have occurred, the chosen forcing rule allows the $i$th force to occur.

A \defterm{complete forcing sequence from \( F \)} is a forcing sequence from \(F\) after which no further forces are possible.
When there is a complete forcing sequence from \(F\) in which the vertices that force are exactly those in \( R \subseteq V(G) \), we say that \( F \) is a set from which \( R \) \defterm{can be the set of vertices that force}.

Note that if $F$ is a zero forcing set, then by definition a complete forcing sequence from \( F \) ends with all vertices filled.
An
\emph{optimal zero forcing set} is a zero forcing set of smallest cardinality
among all zero forcing sets. An \emph{optimal forcing sequence} is a
complete forcing sequence from an optimal zero forcing set.

Every forcing sequence gives rise to a collection of \defterm{forcing chains}, where each forcing chain is a maximal sequence of vertices $(v_0,v_1,\ldots,v_k)$ with the property that, for each $i$ with $1\le i \le k$, the force $v_{i-1} \to v_i$ occurs somewhere in the forcing sequence.  By the maximality of the chain, the vertex $v_0$ must be initially filled.  Since each vertex can force only once, the forcing chains partition the set of vertices that are filled at the point where every force in the forcing sequence has occurred.  We refer to $k$ as the \defterm{length} of the forcing chain.  (So a chain of length $0$ corresponds to an initially filled vertex that never forces.)

The \emph{Hadwiger number} of a graph \( G \),
denoted by \( h(G) \), is the largest \( k \) such that \( G \) contains a \( k
\)-clique as a minor.
We write $\delta(G)$ for the smallest degree of a vertex in $G$, and use $\tw(G)$ to denote the treewidth of $G$, which we define in Section \ref{sec:cobipartite}.

\begin{thm}[{\cite[Figure 1.1]{tree-width-JGT}}]\th\label{treewidthdiagram}
For all graphs \( G \), 
\begin{equation}\label{eqn:parameters from diagram inequality}
\delta(G) \leq \tw(G) \leq Z_+(G) \le \hat Z(G) \le Z(G)
\end{equation}
and
\begin{equation}\label{eqn:M and Z ineq}
h(G)- 1 \leq M_{+}(G) \leq M(G) \le \hat Z(G) \le Z(G).
\end{equation}
\end{thm}

The work of \cite{not_necessarily_symmetric} established a connection between the minimum rank problem for zero-nonzero patterns and the minimum rank problem for graphs in the special case of graphs that satisfy the following definition.

\begin{defn}\thlabel{def:cobipartite graph}
A graph is said to be \defterm{cobipartite} if its complement is bipartite; equivalently, its vertex set can be partitioned into two cliques.
\end{defn}

In particular, for a cobipartite graph with a partition into cliques of sizes $m$ and $n$, an $m\times n$ zero-nonzero pattern can be associated to the graph such that the $\ast$ entries in the pattern occur in positions corresponding to the edges between the two cliques.  A result of \cite{not_necessarily_symmetric} then shows that the minimum rank of the pattern and the maximum nullity of the graph obey a simple relationship (see \thref{thm:not necessarily symmetric min rank result}).  We will show that in this situation, the bounds on those parameters given above, namely the triangle number of the pattern and the zero forcing number of the graph, obey this same relationship.

\section{Triangles and zero forcing sets}\label{sec:triangles-and-forcing-sets}

In this section, we 
explore how the most fundamental combinatorial lower bounds for the minimum rank problem for zero-nonzero matrix patterns and the minimum rank problem for simple graphs, given by the triangle number and the zero forcing number, respectively, share a strong relationship in their underlying combinatorics.
Specifically,
various zero forcing parameters can be given interpretations in terms of the presence, in zero-nonzero patterns associated with the graph, of triangular submatrices satisfying certain constraints.

Recall that the minimum rank problem for graphs is to determine, for a given simple graph $G$, the minimum rank
among all matrices in $\gqual{G}$.  Any given matrix in $\gqual{G}$ has a
specific zero-nonzero pattern; off the diagonal, this must match that of the
adjacency matrix of $G$, but on the diagonal it is unrestricted.  Hence, there are $2^{|G|}$ zero-nonzero patterns possible for $A$, given only that $A \in \gqual{G}$.  For notational convenience, we
define the set of all such zero-nonzero patterns as follows. 

\begin{defn}
For a graph $G$, let $\Sznz(G)$ denote the set
containing the $2^{|G|}$ distinct zero-nonzero patterns belonging to matrices in $\mathcal S(G)$.
\end{defn}

For each zero-nonzero pattern, its triangle number gives a lower bound on the rank of any matrix with that pattern.  Since we are interested in the smallest rank of a matrix in $\gqual{G}$, it is therefore natural to consider what information about this minimum rank can be obtained from the triangle numbers of the individual patterns in $\Sznz(G)$.  One possibility is to consider the minimum triangle number over all patterns in $\Sznz(G)$.  This value represents the best possible lower bound on the minimum rank of $G$ provided by the triangle numbers of the possible patterns alone.
This bound turns out to be exactly the information captured by the enhanced zero forcing number $\hat Z(G)$, as the following theorem shows.

\begin{thm}\th\label{thm:triangles-and-enhanced-ZF}
Let $G$ be a graph.  Then
\[
\hat Z(G) = |G| - \min\{ \tri(\mathcal A) : \mathcal A \in \Sznz(G) \}.
\]
\end{thm}

To prove the above theorem, we need to note that the enhanced zero forcing
number is somewhat different from other zero forcing parameters, in that it is not defined solely in terms of a particular zero forcing rule on the graph
$G$.  Instead, it is defined as the largest zero forcing number
over all loopings $\hat G$ of $G$.  These loopings correspond exactly to the
patterns in $\Sznz(G)$.  Hence, to prove
\th\ref{thm:triangles-and-enhanced-ZF}, it suffices to show that, for each such
looping, its zero forcing number is given by $|G|$ minus the triangle number of the corresponding zero-nonzero pattern.  This fact is a corollary (\thref{cor:triangle number and loop zero forcing relationship} below) of the following theorem, which gives a precise connection between zero forcing sets of a particular looping of $G$ and triangles in the zero-nonzero pattern corresponding to that looping.

\begin{thm}\th\label{thm:triangle-and-forcing-set-for-one-looping}
Let $\hat G$ be a looping of a simple graph $G$, let $\mathcal A$ be the zero-nonzero pattern of the (looped) adjacency matrix of $\hat G$, and let
$R$ and $C$ be subsets of $V(G)$ with $|R|=|C|$.  Then
$\mathcal A[R,C]$ is a triangle if and only if $V(G)\setminus C$ is a loop zero forcing set for $\hat G$ from which $R$ can be the set of vertices that force.
\end{thm}

\begin{proof}
Let $k=|R|=|C|$.  Suppose first that $\mathcal A[R,C]$ is a triangle.  Then a permutation of the rows indexed by $R$ and a permutation of the columns indexed by $C$ exist such that, when these are applied to the submatrix $\mathcal A[R,C]$, the result is a matrix that is lower-triangular with only $*$ entries on its diagonal.

Let $v_1,\ldots,v_k$ be the vertices in $R$, in the order they would occur after this permutation of the rows is applied.  Similarly, let $v_{\pi(1)},\ldots,v_{\pi(k)}$ be the vertices in $C$, again in the order they would occur after the permutation of the columns is applied.

It now suffices to show that, from a state in which $V(G)\setminus C$ is the set of filled vertices, i.e., a state in which $v_{\pi(1)},\ldots,v_{\pi(k)}$ are precisely the vertices of $G$ that are not filled,
\begin{equation}\label{eqn:forces from triangle}
    v_1 \to v_{\pi(1)}, \ldots, v_k \to v_{\pi(k)}
\end{equation}
is a possible forcing sequence.  To this end, we make two observations for each $i \in \{1,\ldots,k\}$.  First, since the permuted submatrix has only $*$ entries on its diagonal, $v_i$ is adjacent to $v_{\pi(i)}$.  Second, since the permuted submatrix has only $0$ entries above its diagonal, $v_i$ is not adjacent to $v_{\pi(j)}$ for any $j > i$.  Hence, for each $i \in \{1,\ldots,k\}$, after the first $i-1$ forces in the list \eqref{eqn:forces from triangle} have occurred, $v_{\pi(i)}$ is the unique unfilled neighbor of $v_i$ in the graph, so that the $i$th force can also occur, which completes this direction of the proof.

Now let $R$ and $C$ be subsets of $V(G)$ such that $V(G)\setminus C$ is a loop zero forcing set from which a forcing sequence exists in which the vertices that force are exactly those in $R$.  In particular, let $v_1,\ldots,v_k$ be the vertices in $R$, listed in the order in which they force, and define $\pi(1),\ldots,\pi(k)$ so that $C=\{v_{\pi(1)},\ldots,v_{\pi(k)}\}$ and so that \eqref{eqn:forces from triangle} gives a forcing sequence from $V(G)\setminus C$.

Then, at the point when the force $v_i \to v_{\pi(i)}$ is about to occur, 
$v_{\pi(i)}$ must be the unique unfilled neighbor of $v_i$.  Since, at that same point, each
$v_{\pi(j)}$ for $j > i$ is unfilled (since it will be forced later on), it follows that $v_i$ is not adjacent to $v_{\pi(j)}$ for any $j > i$.  Hence, ${\mathcal A}_{i\pi(j)}=0$ for $j > i$.  And ${\mathcal A}_{i\pi(i)}=\ast$ holds because $v_i$ is adjacent to $v_{\pi(i)}$.  This shows that $\mathcal A[R,C]$ is a triangle, in particular because applying permutations so that the rows indexed by $v_1,\ldots,v_k$ and the columns indexed by $\pi(1),\ldots,\pi(k)$ occur in the order listed will produce a matrix that is lower-triangular with $\ast$ entries on its diagonal.
\end{proof}

From \th\ref{thm:triangle-and-forcing-set-for-one-looping}, the following corollary follows immediately.

\begin{cor}\th\label{cor:triangle number and loop zero forcing relationship}
Let $G$ be a graph, let $\hat G$ be a looping of $G$, and let
$\mathcal A \in \Sznz(G)$ be the zero-nonzero pattern of $\hat G$.  Then
the zero forcing number of $\hat G$ is given by $|G|-\tri(\mathcal A)$.
\end{cor}

As discussed above, \th\ref{thm:triangles-and-enhanced-ZF} now follows directly from \th\ref{cor:triangle number and loop zero forcing relationship}.
In particular, $\hat Z(G)$, being the largest zero forcing number of a looping of $G$, is seen to equal $|G|$ minus the smallest triangle number of an $\mathcal A \in \Sznz(G)$, which gives \th\ref{thm:triangles-and-enhanced-ZF}.

In contrast to $\hat Z(G)$, the ordinary zero forcing number $Z(G)$ is defined in terms of a zero forcing process that makes no reference to loopings of the graph.  Even so, $Z(G)$ turns out to have a natural interpretation in terms of
the presence of
triangles within
the patterns in $\Sznz(G)$, which correspond to these loopings.  In particular, 
we can understand the zero forcing number in terms of the locations of submatrices that remain triangles regardless of what zero-nonzero pattern occurs on the diagonal.

\begin{defn}\th\label{def:persistent triangle}
Let $G$ be a graph and let $R$ and $C$ be subsets of $V(G)$ with $|R|=|C|$. We say that $(R,C)$ is a \defterm{persistent triangle} of $G$ if $\mathcal A[R,C]$ is a triangle for every $\mathcal A \in \Sznz(G)$.  In this case, we call the value of $|R|=|C|$ the \defterm{size} of the triangle.
\end{defn}

Our next theorem shows how persistent triangles of $G$ are related to zero forcing sets of $G$.  Before stating this theorem, we present the following example to illustrate this relationship.

\begin{ex}\thlabel{ex:zero forcing set and persistent triangle}
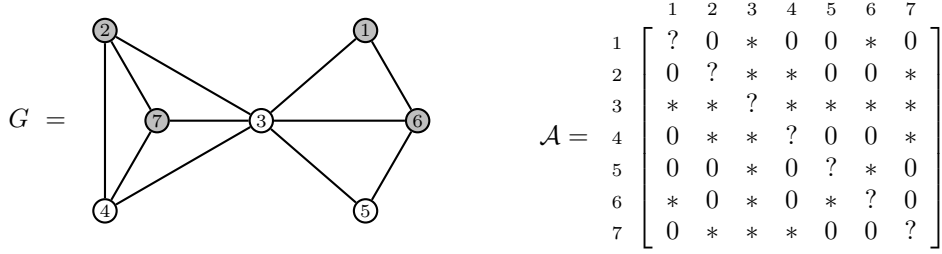
\begin{figure}
    \centering
$G~=$\quad\begin{minipage}{2in}
    \begin{tikzpicture}[thick, main node/.style={circle,minimum width=7pt,draw,inner sep=1pt}, filled node/.style={fill={lightgray},circle,minimum width=7pt,draw,inner sep=1pt},dark node/.style={fill={black},circle,minimum width=7pt,draw,inner sep=1pt}]
\def \sf {1.2}
\draw
    (-1.732*\sf,1*\sf) node[filled node] (a) {\scriptsize $2$}
    (-1.1547*\sf,0*\sf) node[filled node] (b) {\scriptsize $7$}
    (-1.732*\sf,-1*\sf) node[main node] (c) {\scriptsize $4$}
    (0*\sf,0*\sf) node[main node] (d) {\scriptsize $3$}
    (1.1547*\sf,1*\sf) node[filled node] (e) {\scriptsize $1$}
    (1.1547*\sf,-1*\sf) node[main node] (f) {\scriptsize $5$}
    (1.732*\sf,0*\sf) node[filled node] (g) {\scriptsize $6$}
    ;
\path
    (a) edge (b)
    	edge (c)
		edge (d)
    (b) edge (c)
    	edge (d)
    (c) edge (d)
    (d) edge (e)
    (d) edge (f)
        edge (g)
    (e) edge (g)
    (f) edge (g);
\end{tikzpicture}
\end{minipage}
\hspace{0.25in}
\begin{minipage}{2.25in}
       $\mathcal A =~ \kbordermatrix{
        ~ & 1 & 2 & 3 & 4 & 5 & 6 & 7 \\
        1 & ? & 0 & * & 0 & 0 & * & 0  \\
        2 & 0 & ? & * & * & 0 & 0 & *  \\
        3 & * & * & ? & * & * & * & *  \\
        4 & 0 & * & * & ? & 0 & 0 & *  \\
        5 & 0 & 0 & * & 0 & ? & * & 0  \\
        6 & * & 0 & * & 0 & * & ? & 0  \\
        7 & 0 & * & * & * & 0 & 0 & ? 
        }$
\end{minipage}
    \caption{For the graph $G$ above, every $\mathcal A\in\Sznz(G)$ has the form given, where the symbol ``$?$''\ denotes an entry that could be either $0$ or $*$.}
    \label{fig:graph with associated triangles}
\end{figure}

For the graph $G$ shown in Figure \ref{fig:graph with associated triangles}, any $\mathcal A \in \Sznz(G)$ has the form given there.  Now note that $\{1,2,6,7\}$ is a zero forcing set for $G$, from which a possible sequence of forces is $1 \to 3, 2\to 4, 3\to 5$.  In this sequence, the set of vertices that force is $\{1,2,3\}$.  The set of vertices that get forced necessarily has the same size, and is $\{3,4,5\}$.  Hence, letting $R=\{1,2,3\}$ and $C=\{3,4,5\}$, we have that $V(G)\setminus C$ is a zero forcing set for $G$ from which $R$ can be the set of vertices that force.

Now note that the submatrix of $\mathcal A$ with rows indexed by $R$ and columns indexed by $C$ is
\[
\kbordermatrix{
    ~ & 3 & 4 & 5 \\
    1 & * & 0 & 0 \\
    2 & * & * & 0 \\
    3 & ? & * & * 
}.
\]
In particular, this submatrix is a triangle regardless of whether the entry in its bottom-left corner (a diagonal entry from $\mathcal A$) is $*$ or $0$.  In particular, then, $\mathcal A[R,C]$ is a triangle for every $\mathcal A \in \Sznz(G)$.  Per \thref{def:persistent triangle}, this means that $(R,C)$ is a persistent triangle of $G$.
\end{ex}

The next theorem states that the relationship illustrated by \thref{ex:zero forcing set and persistent triangle}, between zero forcing sets of a graph $G$ and persistent triangles of $G$, holds in general.

\begin{thm}\th\label{thm:zero forcing sets and persistent triangles}
Let $G$ be a graph and let $R$ and $C$ be subsets of $V(G)$ with $|R|=|C|$.  Then
$(R,C)$ is a persistent triangle of $G$ if and only if $V(G)\setminus C$ is a zero forcing set for $G$ from which $R$ can be the set of vertices that force.
\end{thm}

We prove \thref{thm:zero forcing sets and persistent triangles} by first establishing the following lemma.  It says that being a zero forcing set for a graph is equivalent to being a zero forcing set for every looping of the graph, and that this equivalence preserves the set of vertices that do the forcing.

\begin{lem}\th\label{lem:ordinary-ZF-is-equivalent-to-ZF-for-every-looping}
Let $G$ be a graph with $F,R\subseteq V(G)$.  The following are equivalent.
\begin{enumerate}
    \item\label{cond:F set R forces ordinary forcing}
    The set $F$ is a zero forcing set for $G$ from which $R$ can be the set of vertices that force.
    \item \label{cond:F set R forces every looping}
    For every looping $\hat G$ of $G$, the set $F$ is a zero forcing set from which $R$ can be the set of vertices that force.
\end{enumerate}
\end{lem}

\begin{proof}
Suppose first that \ref{cond:F set R forces ordinary forcing} holds.  Let $\hat G$ be any looping of $G$.  Under loop zero forcing, the presence or absence of a loop on a vertex has no effect on its ability to force when it is filled.  As a result, any vertex that could force under the ordinary zero forcing rule could also do so under the loop zero forcing rule, regardless of the looping. Hence, $F$ is a loop zero forcing set for $\hat G$ from which the exact same vertices, namely those in $R$, can be taken to force, so that \ref{cond:F set R forces every looping} holds.

Now assume that \ref{cond:F set R forces ordinary forcing} does not hold.  
Then, starting with $F$ as the set of filled vertices and letting, at each step, some vertex of $R$ force until no further such forces are possible, we obtain a set $F'$ with $F \subseteq F' \subset V(G)$
such that from $F'$ no vertex in $R$ can force.
We use $F'$ to construct a looping $\hat G$ that contradicts \ref{cond:F set R forces every looping}.

In $\hat G$, each vertex with a loop will be in $R$, and, in particular, $x\in R$ receives a loop if and only if $x$ is unfilled and has exactly one unfilled neighbor in $G$ (when $F'$ is the set of filled vertices).  This leaves each such $x$ with two unfilled neighbors in $\hat G$, one of which is its own loop.
Hence, in $\hat G$ every unfilled vertex in $R$ has either zero or at least two unfilled neighbors.  So no unfilled vertex in $R$ can force in $\hat G$.
At the same time, no filled vertex of $R$ can force in $\hat G$ either, since no ordinary force was possible from this vertex in $G$ (again, with $F'$ as the set of filled vertices).

By the above, it is not possible for $F'$ to force all of $\hat G$ with only vertices in $R$ forcing.  Since $F \subseteq F'$, this must be true of $F$ as well.  So \ref{cond:F set R forces every looping} does not hold.
\end{proof}

\begin{proof}[Proof of \thref{thm:zero forcing sets and persistent triangles}]
By \th\ref{def:persistent triangle}, the condition that  $(R,C)$ is a persistent triangle of $G$ means that $\mathcal A[R,C]$ is a triangle for every $\mathcal A \in \Sznz(G)$. By \th\ref{thm:triangle-and-forcing-set-for-one-looping}, this is equivalent to saying that for every looping $\hat G$, the set $V(G)\setminus C$ is a zero forcing set for $\hat G$ from which $R$ can be the set of vertices that force.
By \th\ref{lem:ordinary-ZF-is-equivalent-to-ZF-for-every-looping}, this is in turn equivalent to the statement that $V(G)\setminus C$ is a zero forcing set for $G$ from which $R$ can be the set of vertices that force.
\end{proof}

The persistent triangles of $G$ are defined in terms of triangles in the patterns in $\Sznz(G)$.  The following corollary shows that $Z(G)$ can be understood entirely in these terms.

\begin{cor}\th\label{cor:zero forcing and persistent triangles}
Let $G$ be a graph.  Then $|G|-Z(G)$ is the size of a largest persistent triangle of $G$.
\end{cor}

\begin{proof}
Let $k$ be the size of a largest persistent triangle of $G$.
Let $F$ be a zero forcing set with $|F|=Z(G)$.
Let $C=V(G)\setminus F$. \th\ref{thm:zero forcing sets and persistent triangles} says that $G$ has a persistent triangle of size $|C|=|G|-Z(G)$.  So  $k \ge |G|-Z(G)$.

For the reverse inequality, let $(R,C)$ be a persistent triangle of $G$ with $|R|=|C|=k$.  Then \th\ref{thm:zero forcing sets and persistent triangles} gives that $V(G)\setminus C$ is a zero forcing set for $G$ of size $|G|-k$.  So, $|G|-k \ge Z(G)$.  Thus, $k \le |G|-Z(G)$.
\end{proof}

As \th\ref{cor:zero forcing and persistent triangles} shows, $Z(G)$ precisely captures the best bound on $M(G)$ that can be witnessed by some fixed indices locating a submatrix that is a triangle within every pattern in $\Sznz(G)$.  On the other hand, $\hat Z(G)$ captures the possibility that the largest triangle within each pattern from $\Sznz(G)$ may be in a different location depending upon the pattern of the diagonal.  In particular, when a gap occurs such that $\hat Z(G) < Z(G)$, this is precisely because every pattern in $\Sznz(G)$ has a triangle that is larger than the largest persistent triangle (so that necessarily these triangles do not all occur in the same location).

Beyond requiring that a fixed location give a submatrix that is a triangle regardless of the zero-nonzero pattern of the diagonal entries, we may require that the location excludes diagonal entries altogether.  This stronger notion of a triangle associated with $G$ can be formalized as follows. 

\begin{defn}\th\label{def:immutable triangle}
Let $G$ be a graph and let $R$ and $C$ be subsets of $V(G)$ with $|R|=|C|$.
When $(R,C)$ is a persistent triangle of $G$ and $R\cap C = \emptyset$, we say that $(R,C)$ is an \defterm{immutable triangle} of $G$.  In particular, this occurs when, for every $\mathcal A \in \Sznz(G)$, the submatrix $\mathcal A[R,C]$ is a triangle  with no entries coming from the diagonal of $\mathcal A$.  Again in this case, we refer to the value of $|R|=|C|$ as the \defterm{size} of the immutable triangle.
\end{defn}

In the same way that persistent triangles associated with $G$ correspond to zero forcing sets, those
that are
immutable triangles correspond to zero forcing sets of a special type.

\begin{defn}\th\label{def:direct ZF set}
Let $G$ be a graph.  We say that
a zero forcing set $F$ for $G$ is a \defterm{direct zero forcing set} if there is a complete forcing sequence from $F$ such that all vertices that force are in $F$.
\end{defn}

The following observation follows easily from \th\ref{def:direct ZF set}.

\begin{obs}\th\label{obs:directly forceable chain length}
Let $F$ be a zero forcing set for a graph.  Then $F$ is a direct zero forcing set if and only if there is a complete forcing sequence from $F$ with each of the resulting forcing chains of length $0$ or $1$.
\end{obs}

\begin{defn}
Let $G$ be a graph.  The \defterm{direct zero forcing number} of $G$, denoted by $Z_d(G)$, is the smallest size of a direct zero forcing set for $G$.
\end{defn}

\begin{defn}
We say that a graph $G$ is \defterm{directly forceable} if $Z_d(G)=Z(G)$.
\end{defn}

The \emph{propagation time} for a graph $G$ is defined in
\cite[Definition 1.3]{prop_time_paper}.    It follows from
\th\ref{obs:directly forceable chain length} that the class of directly
forceable graphs includes all graphs with propagation time
$1$, a class
investigated in \cite{prop_time_paper}.  The following shows that the directly forceable graphs give a strict generalization of this class.

\begin{figure}
    \centering
$G_1~=$\quad\begin{minipage}{1.5in}
    \begin{tikzpicture}[thick, main node/.style={circle,minimum width=7pt,draw,inner sep=1pt}, filled node/.style={fill={lightgray},circle,minimum width=7pt,draw,inner sep=1pt}]
    \def \sf {1.2}
    \draw
        (0*\sf,3*\sf) node[filled node] (1) {\scriptsize 1}
        (0*\sf,2*\sf) node[filled node] (2) {\scriptsize 2}
        (0*\sf,1*\sf) node[filled node] (3) {\scriptsize 3}
        (2*\sf,3*\sf) node[main node] (4) {\scriptsize 4}
        (2*\sf,2*\sf) node[main node] (5) {\scriptsize 5}
        (2*\sf,1*\sf) node[main node] (6) {\scriptsize 6};
    \path
        (2) edge (1)
            edge (6)
            edge (5)
        (1) edge (6)
            edge (5)
        (3) edge (6)
        (5) edge (6)
            edge (4)
        (4) edge (1);
    \end{tikzpicture}
\end{minipage}
\hspace{0.125in}
$G_2~=$\quad\begin{minipage}{1.5in}
    \begin{tikzpicture}[thick, main node/.style={circle,minimum width=7pt,draw,inner sep=1pt}, filled node/.style={fill={lightgray},circle,minimum width=7pt,draw,inner sep=1pt}]
    \def \sf {1.2}
    \draw
        (0*\sf,3*\sf) node[filled node] (1) {\scriptsize 1}
        (0*\sf,2*\sf) node[filled node] (2) {\scriptsize 2}
        (0*\sf,1*\sf) node[filled node] (3) {\scriptsize 3}
        (2*\sf,3*\sf) node[main node] (4) {\scriptsize 4}
        (2*\sf,2*\sf) node[main node] (5) {\scriptsize 5}
        (2*\sf,1*\sf) node[main node] (6) {\scriptsize 6};
    \path
        (2) edge (1)
            edge (6)
            edge (5)
            edge (3)
        (1) edge (6)
            edge (5)
            edge[bend right] (3)
        (3) edge (6)
        (5) edge (6)
            edge (4)
        (4) edge (1)
            edge[bend left] (6);
    \end{tikzpicture}
\end{minipage}
\hspace{0.25in}
\begin{minipage}{1.5in}
       $\Y = \kbordermatrix{
        ~ & 4 & 5 & 6 \\
        1 & * & * & * \\
        2 & 0 & * & * \\
        3 & 0 & 0 & * \\
        }$
\end{minipage}
    \caption{The set \(\{1,2,3\}\) is a direct zero forcing set for $G_1$ and for $G_2$.  Each graph is
    partitioned according to the pattern $\Y$ (see \thref{def:graph partitioned according to pattern}).}
    \label{fig:graph partitioned with pattern}
\end{figure}
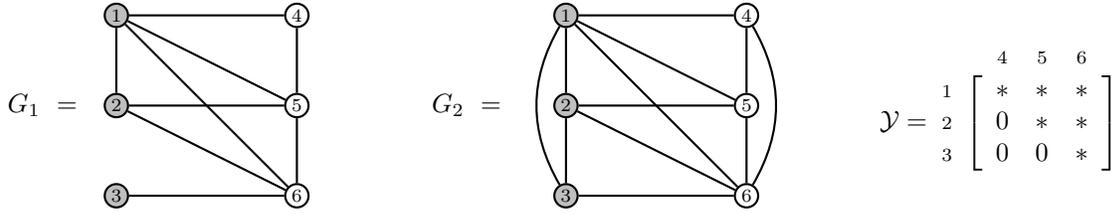

\begin{ex}
Let \( G \) be either of the graphs \( G_1 \) and \( G_2 \) shown in Figure \ref{fig:graph partitioned with pattern}. Then \( F = \{ 1,2,3 \} \) is a zero forcing set for \( G \) from which
\( 3 \to 6, 2 \to 5, 1 \to 4 \) is a complete forcing sequence.  In that sequence, only vertices in $F$ force, showing that $F$ is a direct zero forcing set for $G$.
As there are no zero
forcing sets of smaller size, \( F \) is an optimal zero forcing
set, and hence \( G \) is directly forceable. Note, however, that by \cite[Proposition 3.12]{prop_time_paper}, the propagation time of \( G
\) is not \( 1 \).
\end{ex}

\begin{thm}\th\label{thm:immutable triangles and direct forcing sets}
Let $G$ be a graph and let $R$ and $C$ be subsets of $V(G)$ with $|R|=|C|$.  Then
$(R,C)$ is an immutable triangle of $G$ if and only if $V(G)\setminus C$ is a direct zero forcing set for $G$ from which $R$ can be the set of vertices that force.
\end{thm}

\begin{proof}
Let $F=V(G)\setminus C$.
First assume $(R,C)$ is an immutable triangle of $G$.  
By definition, $(R,C)$ is a persistent triangle of $G$ as well.
It follows by \th\ref{thm:zero forcing sets and persistent triangles} that
$F$ is a zero forcing set for $G$ from which $R$ can be the set of vertices that force.  Since $(R,C)$ is an immutable triangle of $G$, we have $R \cap C = \emptyset$, so that
$R \subseteq V(G)\setminus C = F$.  That is, only vertices in $F$ force, and therefore $F$ is a direct zero forcing set.

Now suppose $F$ is a direct zero forcing set for $G$ from which $R$ can be the set of vertices that force.  \th\ref{thm:zero forcing sets and persistent triangles} gives that $(R,C)$ is a persistent triangle of $G$.  Since $F$ is direct zero forcing set, only vertices in $F$ force.  That is, $R \subseteq F = V(G)\setminus C$, and hence $R \cap C = \emptyset$, so that $(R,C)$ is an immutable triangle of $G$.
\end{proof}

The following corollary now follows from \th\ref{thm:immutable triangles and direct forcing sets} exactly as \th\ref{cor:zero forcing and persistent triangles} follows from \th\ref{thm:zero forcing sets and persistent triangles}.

\begin{cor}
Let $G$ be a graph.  Then $|G|-Z_d(G)$ is the size of a largest immutable triangle of $G$.
\end{cor}

When a zero forcing set $F$ for a graph $G$ forces via a certain set of forcing chains, the set of
vertices that occur as the first vertex in some chain is just $F$.  Hence, $F$ represents a way of choosing one endpoint from each chain.  If we exchange each vertex of $F$ with the opposite endpoint of its chain, the resulting set $F'$ is called a \defterm{reversal} of $F$ \cite[Definition 2.5]{ZF_parameters_and_min_rank}; we say that $F'$ is the reversal of $F$ with respect to that particular sequence of forces.
Note that there may be vertices that are in both $F$ and $F'$, namely any vertices in chains of length $0$.
It is not difficult to see that $F'$ is also a zero forcing set for $G$
\cite[Theorem 2.6]{ZF_parameters_and_min_rank}
with $|F'|=|F|$.

When $F$ is a direct zero forcing set, and $F'$ is the reversal of $F$ with respect to some sequence of forces as given by \th\ref{obs:directly forceable chain length}, it follows from that observation that $F'$ is also a direct zero forcing set.  Moreover, if we let $R$ be
the set of vertices occurring as the first vertices of the (original) forcing chains of length $1$, then we first observe that $R$ is precisely the set of vertices that at some point force during the sequence of forces from $F$ that we originally considered.  We next observe that $F'=V(G)\setminus R$.
The above observations suffice to prove the following lemma, which we will find useful in Section \ref{sec:cobipartite}.

\begin{lem}\th\label{lem:reversal of direct ZF set}
Let $F$ be a direct zero forcing set for a graph $G$.  In some complete forcing sequence from $F$ with each of the resulting forcing chains of length $0$ or $1$, let $R$ be the set of vertices that force.  Then  $V(G) \setminus R$ gives a reversal of $F$ that is also a direct zero forcing set for $G$ of size $|F|$.
\end{lem}

\section{Triangle numbers of patterns and zero forcing numbers of graphs}\label{sec:triangles and forcing of partitioned graphs}

The results of the previous section allow us to start with a graph $G$ and understand different zero forcing bounds for $G$ in terms of the presence of specific types of triangles within the patterns associated to $G$, namely those in $\Sznz(G)$.  In this section, we explore connections in the other direction; that is, starting with an arbitrary zero-nonzero pattern $\Y$, we examine how a graph $G$ may have the property that the triangles within $\Y$ correspond to specific zero forcing sets of $G$.  One construction of such a graph recalls the way in which a bipartite graph is naturally constructed from a $(0,1)$-matrix; constructing a cobipartite graph from $\Y$ in a similar way (see \thref{def:cobipartite associated to pattern}) gives the sort of connection described above.  We will see, however, that this cobipartite graph is only one among a more general class of graphs that share the desired connection with $\Y$.  \thref{cor:min ZF set from left to right} that follows allows us to study these graphs more generally, while the next section, Section \ref{sec:cobipartite}, examines the cobipartite case in more detail.

\begin{defn}\th\label{def:graph partitioned according to pattern}
Let $\Y$ be an $m\times n$ zero-nonzero pattern, let $G$ be a graph on $m+n$ vertices, and let
$V_r=\{r_1,\ldots,r_m\}$
and
$V_c=\{c_1,\ldots,c_n\}$
be disjoint subsets of $V(G)$ with $V(G)=V_r\cup V_c$.
We say that $(V_r,V_c)$ is a \defterm{partition of $G$ according to $\Y$}
if there is an edge from $r_i$ to $c_j$ in $G$ exactly when the $(r_i,c_j)$-entry of $\Y$ is nonzero.
\end{defn}

\begin{ex}\th\label{ex:graph partitioned according to pattern}
Let \( G \) be either of the graphs \( G_1 \) and \( G_2 \) shown in Figure \ref{fig:graph partitioned with pattern}.
Let $V_1=\{1,2,3\}$ and $V_2=\{4,5,6\}$.  Then $(V_1,V_2)$ is a partition of $G$ according to $\Y$.  Intuitively, this means that the pattern $\Y$ can be taken to describe precisely the edges of the graph that have one endpoint in $V_1$ and the other in $V_2$.  The other edges, those with both endpoints in $V_1$ or both endpoints in $V_2$, have no bearing on whether or not $(V_1,V_2)$ is a partition of the graph according to $\Y$.
\end{ex}

When a graph $G$ does have a partition according to $\Y$, there is a relationship between the triangles in $\Y$ and certain zero forcing sets for $G$.

\begin{defn}\thlabel{def:forcing from X to Y}
Let $G$ be a graph, let $V_1$ and $V_2$ be subsets of $V(G)$, and let $F$ be a
zero forcing set for $G$.  We say that $F$ 
\defterm{forces from $V_1$ to $V_2$} if
$V_1 \subseteq F$ and
there exists a complete forcing sequence from $F$ such that every vertex that forces is in $V_1$ and every vertex that gets forced is in $V_2$.
\end{defn}

\begin{obs}\thlabel{forcing set V1 to V2 must be direct}
A zero forcing set $F$ can satisfy \th\ref{def:forcing from X to Y} only when it is a direct zero forcing set, since the definition implies that every vertex that forces is an element of $F$.
\end{obs}

Note that, for any graph $G$, and any arbitrary subsets $V_1$ and $V_2$ of its vertices, there does exist at least one zero forcing set $F$ for $G$ that forces from $V_1$ to $V_2$, since one can always take $F=V(G)$.  Hence, the minimum to which the following theorem refers is always well-defined.

\begin{thm}\th\label{thm:triangles and sets forcing X to Y}
Let $\Y$ be an $m\times n$ zero-nonzero pattern and let $G$ be a graph with a partition $(V_1,V_2)$ according to $\Y$.
Let $k$ be the smallest size of a zero forcing set $F$ for $G$ that forces from $V_1$ to $V_2$.  Then
\[
    \tri(\Y) = m+n-k.
\]
\end{thm}

\begin{proof}
Letting $F$ be a zero forcing set that forces from $V_1$ to $V_2$, there is some complete forcing sequence from $F$ such that, with $R$ the set of vertices in the sequence that force and $C$ the set of those that get forced (so that $C=V(G)\setminus F$), we have $R \subseteq V_1$ and $C \subseteq V_2$.  Then $(R,C)$ is a triangle, of size $m+n-|F|$, within every matrix in $\Sznz(G)$, by \thref{thm:zero forcing sets and persistent triangles}. Since $(V_1,V_2)$ is a partition of $G$ according to $\Y$, the fact that $R \subseteq V_1$ and $C \subseteq V_2$ implies that this corresponds to a triangle  within $\Y$.  Choosing $F$ with size $k$ thus gives $\tri(\Y) \ge m+n-k$.

Now suppose that a largest triangle within $\Y$ has rows indexed by $R$ and columns indexed by $C$.
Then $\tri(\Y)=|C|$.
Since $(V_1,V_2)$ is a partition of $G$ according to $\Y$, we have $R \subseteq V_1$ and $C \subseteq V_2$.  Since the triangle is within $\Y$, it is an immutable triangle. 
Hence, letting $F = V(G)\setminus C$, we have by \th\ref{thm:immutable
triangles and direct forcing sets} that $F$ is a zero forcing set for $G$ from which $R$ can be the set of vertices that force.  And, of course, the set of
vertices that get forced is simply $V(G)\setminus F = C$.  Since $R \subseteq
V_1$ and $C \subseteq V_2$, it follows that $F$ forces from $V_1$ to
$V_2$.  Thus, $k \le |F|$.
The above now yields
\[
    \tri(\Y) = |C|=|V(G)\setminus F|=m+n-|F| \le m+n-k. \qedhere
\]
\end{proof}

\thref{thm:zero forcing sets and persistent triangles} showed how the presence of a triangle in a fixed location within every pattern of $\Sznz(G)$ is equivalent to the existence of a corresponding zero forcing set for $G$.  \thref{thm:triangles and sets forcing X to Y} gives a connection between zero forcing sets for a graph and triangles in a  different way, by showing that when the vertex set of $G$ is partitioned such that the pattern $\Y$ gives the edges crossing this partition (\thref{def:graph partitioned according to pattern}), the largest size of a triangle within $\Y$ is equal to the smallest size of a zero forcing set for $G$ that forces from one side of the partition to the other (\thref{def:forcing from X to Y}).  Therefore, when some \emph{optimal} zero forcing set meets this restriction, we can express the value of $Z(G)$ in terms of the size of a largest triangle within $\Y$.  In particular, we have the following immediate corollary of \thref{thm:triangles and sets forcing X to Y}.

\begin{cor}\th\label{cor:min ZF set from left to right}
Let $\Y$ be an $m\times n$ zero-nonzero pattern and let $(V_1,V_2)$ be a partition of a graph $G$ according to $\Y$.
If there exists an optimal
zero forcing set $F$ for $G$ that forces from $V_1$ to $V_2$, then
\begin{equation}\label{eqn:most general Z vs triangle number}
    Z(G) = m+n-\tri(\Y).    
\end{equation}
\end{cor}

\begin{ex}
Let the graph $G$, the pattern $\Y$, and \( V_1, V_2 \subseteq V(G) \) be as in \thref{ex:graph partitioned according to pattern}.  Then
it is straightforward to check that $Z(G)=3$, and
$F=\{1,2,3\}$ is an optimal zero forcing set for $G$ that forces from $V_1$ to $V_2$.  Thus, by \thref{cor:min ZF set from left to right}, the triangle number of $\Y$ must satisfy $3+3-\tri(\Y)=Z(G)=3$.  Indeed, it is easy to see that in fact $\tri(\Y)=3$.
\end{ex}

Given a zero-nonzero pattern $\Y$, we may wish to find a graph $G$
satisfying the hypotheses of \th\ref{cor:min ZF set from left to right}, so that
\eqref{eqn:most general Z vs triangle number} will apply.
The requirement that
$G$ has a partition $(V_1,V_2)$ according to $\Y$
determines the edges between $V_1$ and $V_2$, but then edges within each $V_i$ must be chosen to ensure that
some optimal zero forcing set forces from $V_1$ to $V_2$.
In the next section, we show that one way to achieve this is to include all possible edges within each of $V_1$ and $V_2$, thus producing a cobipartite graph.  
In this case, the relationship turns out to be very strong between the pattern $\Y$ and the graph $G$, both in terms of their minimum rank (or maximum nullity) parameters, and in terms of the combinatorial bounds on those parameters.

\section{Combinatorial bounds for cobipartite graphs and patterns}\label{sec:cobipartite}

Recall from \thref{def:cobipartite graph} that a graph is said to be cobipartite when its complement is bipartite or, equivalently, when it is possible to partition its vertex set into two cliques.
In the special case where this gives a partition of the graph according to some specific zero-nonzero pattern, it was shown in \cite{not_necessarily_symmetric} that (over an infinite field, and subject to a mild nondegeneracy condition; see \thref{def:saturated}) the minimum rank of the graph and the minimum rank of the zero-nonzero pattern must be equal.
The main result of this section, \thref{thm:triangle number versus ZF for cobipartite}, shows that in this same situation the combinatorial bounds on those two parameters obey the same relationship.

\begin{defn}\label{def:clique partition}
Let $G$ be a graph.
A \defterm{clique partition} $(V_1,V_2)$ of $G$ is a partition of the vertex set of $G$ into two cliques.
(Note that this exists exactly when $G$ is cobipartite.)    
\end{defn}

\begin{defn}\th\label{def:cobipartite associated to pattern}
Let $\Y$ be an $m\times n$ zero-nonzero pattern and $G$ be a graph on $m+n$ vertices.  We say that $G$ is the \defterm{cobipartite graph associated with $\Y$}
if there exists a clique partition $(V_r,V_c)$ of $G$ that is a partition of $G$ according to $\Y$ (see \thref{def:graph partitioned according to pattern}).
\end{defn}

\begin{ex}\thlabel{ex:cobipartite associated to pattern}
Both of the graphs $G_1$ and $G_2$ shown in Figure \ref{fig:graph partitioned with pattern} are partitioned according to the pattern $\Y$ given there.
The graph $G_1$ is not cobipartite; the set $\{2,3,4\}$ of vertices induces an odd cycle in its complement.
The graph $G_2$ is cobipartite; both $(\{1,2,3\},\{4,5,6\})$ and $(\{2,3\},\{1,4,5,6\})$ are clique partitions of the graph.  The former is a partition of $G_2$ according to $\Y$, and hence $G_2$ is the cobipartite graph associated with $\Y$.
\end{ex}

In comparison to the standard way in which a bipartite graph is associated with a rectangular $(0,1)$-matrix, the cobipartite graph associated with $\Y$ is the complement of the bipartite graph that would be associated with the $(0,1)$-matrix whose zero-nonzero pattern is complement of $\Y$ (i.e., the pattern obtained from $\Y$ by replacing each of its $*$ entries with $0$ and vice versa).

\begin{defn}\thlabel{def:saturated}
Let $G$ be a graph.  We say that $G$ is \emph{saturated} with respect to the clique partition $(V_1,V_2)$ of $G$ if every vertex in $V_1$ is adjacent to some vertex in $V_2$, and vice versa.
\end{defn}

\begin{ex}
As noted in \thref{ex:cobipartite associated to pattern}, the graph $G_2$ of Figure \ref{fig:graph partitioned with pattern} has both $(\{1,2,3\},\{4,5,6\})$ and $(\{2,3\},\{1,4,5,6\})$ as clique partitions.  It is easy to verify that $G_2$ is saturated with respect to the first of these partitions, but with respect to the second it is not, since vertex $4$ has no neighbor in the set $\{2,3\}$.
\end{ex}

\subsection{Triangles and zero forcing sets for cobipartite graphs}

The following lemma shows that when a cobipartite graph is saturated with respect to some clique partition, \thref{cor:min ZF set from left to right} can be applied, giving the powerful connection of \thref{thm:triangle number versus ZF for cobipartite}, which follows.

\begin{lem}\th\label{normalized forcing for cobipartite}
Let $G$ be a cobipartite graph that is saturated with respect to the clique partition $(V_1,V_2)$.  Then there is an optimal zero forcing set $F$ for $G$ such that $F$ forces from $V_1$ to $V_2$.
\end{lem}

\begin{proof}
By \thref{def:forcing from X to Y}, we must show that there is an optimal zero forcing set $F$ for $G$ with the following two properties.  The first is that $V_1 \subseteq F$.  The second is that some complete forcing sequence from $F$ has only vertices in $V_1$ force.  (Note that, since $V_2=V(G)\setminus V_1$, the fact that $V_1 \subseteq F$ implies that every vertex that gets forced is in $V_2$).

We begin by showing that we can achieve the first property.  To this end, let $F$ be an optimal zero forcing set for $G$, and let $v\to w$ be the first force in some complete forcing sequence from $F$.
Then $w$ is the only neighbor of $v$ that is initially unfilled.
We first proceed under the assumption that $v\in V_1$.  We will revisit this assumption later in the proof.

We claim that either $V_1 \subseteq F$ already holds, or we can modify $F$ so as to make it hold.  For this, we consider two cases.  First, suppose $w \in V_2$.  Then, since every vertex of $V_1\setminus\{v\}$ is a neighbor of $v$, every vertex in $V_1$ is initially filled.  That is, $V_1 \subseteq F$.

Now suppose $w \in V_1$.  Then all of $V_1\setminus \{w\}$ must be initially filled.  Since $G$ is saturated with respect to the clique partition $(V_1,V_2)$, there is some $w' \in V_2$ that is adjacent to $v$.  In order for $v$ to force $w$, it must be that $w'$ is initially filled.  So let $F' = (F\setminus\{w'\})\cup\{w\}$.  With $F'$ as the set of initially filled vertices, $w'$ is the unique unfilled neighbor of $v$.  Hence, a forcing a sequence from $F'$ can begin with $v$ forcing $w'$, and proceed after that point as did the forcing sequence from $F$.  
(This is because, after the first force in each sequence, the set of filled vertices is the same, namely $F\cup\{w\}=F'\cup\{w'\}$.)
Moreover, $V_1 \subseteq F'$.  So, since $|F'|=|F|=Z(G)$, we may let $F'$ replace $F$ and again achieve $V_1 \subseteq F$.

As shown above, we can choose $F$ such that $V_1 \subseteq F$. 
We now show that some complete forcing sequence from $F$ has only vertices in $V_1$ force.  
Consider any complete forcing sequence from $F$.  
If every vertex that forces in this sequence is in $V_1$, then there is nothing to show.  Otherwise, some force $x\to y$ in the sequence has $x \in V_2$.  Take this to be the first such force.  As $V_1 \subseteq F$, it must be that $y \in V_2$.  The fact that $V_2$ is a clique implies that every vertex in $V_2$ is adjacent to $x$, and hence every vertex in $V_2\setminus\{y\}$ must be filled before $x$ can force $y$.  Since every vertex in $V_1$ is initially filled, this implies that, at the time the force is to occur, every vertex of the graph other than $y$ is filled.  But, since $G$ is saturated with respect to the clique partition $(V_1,V_2)$, some $x' \in V_1$ is adjacent to $y$.  Thus, we can replace the force $x\to y$ in the sequence with $x'\to y$.  By the way the force was chosen, no vertex in $V_2$ forces earlier in the sequence.  And after the force is complete, all vertices in the graph are filled, so there are no subsequent forces.  Hence, in the modified sequence, no vertices in $V_2$ force.  That is, the sequence has only vertices in $V_1$ force, as desired.

Finally, we assumed at the start of the proof that $v \in V_1$.  If instead $v \in V_2$, then a symmetric argument shows that $F$ can be modified if necessary to give
$V_2 \subseteq F$, with a complete forcing sequence from $F$ such that only vertices in $V_2$ force.
Since every vertex forced
is forced by a vertex that was initially filled, every forcing chain has a length of $0$ or $1$, and $F$ is a direct zero forcing set for $G$.
Let $R$ be the set of vertices that force, and let $F'=V(G)\setminus R$.
Then, by \th\ref{lem:reversal of direct ZF set}, the set $F'$ is
a zero forcing set for $G$ with $|F'|=|F|$.
As the forcing sequence from $F$ has only vertices in $V_2$ force, $R\subseteq V_2$.  Thus, $F' = V(G)\setminus R \supseteq V(G)\setminus V_2 = V_1$.
Hence, $F'$ is an optimal zero forcing set for $G$ with $V_1\subseteq F'$.  We may therefore let $F'$ replace $F$ and, following the argument in the previous paragraph, deduce that some complete forcing sequence from $F$ has only vertices in $V_1$ force, as desired.
\end{proof}

\begin{ex}\th\label{ex:saturated is necessary}
The graph shown in Figure \ref{fig:first unsaturated example} shows that the condition of $G$ being saturated is essential for \th\ref{normalized forcing for
cobipartite}.  For $V_1=\{1,2,3\}$ and $V_2=\{4,5,6\}$, a clique partition of $G$ is given by \( (V_1,V_2) \), but $G$ is not saturated with respect to this partition.  Moreover, while in fact
$Z(G)=3$, clearly no zero forcing set of this size can
force from $V_1$ to $V_2$, as such a set would necessarily be equal to $V_1$, and clearly then would fail to force the entire graph.
\end{ex}

\begin{figure}[ht]
    \centering
\begin{tikzpicture}[thick, main node/.style={circle,minimum width=7pt,draw,inner sep=1pt}, filled node/.style={fill={lightgray},circle,minimum width=7pt,draw,inner sep=1pt}]
\def \sf {1.2}
\draw
    (0*\sf,1*\sf) node[main node] (3) {\scriptsize 3}
    (0*\sf,-1*\sf) node[filled node] (4) {\scriptsize 4}
    (0.6*\sf,0*\sf) node[main node] (5) {\scriptsize 5}
    (1.56*\sf,0*\sf) node[main node] (6) {\scriptsize 6}
    (-0.6*\sf,0*\sf) node[filled node] (2) {\scriptsize 2}
    (-1.56*\sf,0*\sf) node[filled node] (1) {\scriptsize 1}
    ;
\draw[color=gray,dashed] (-1.25,-1.25) -- (1.25*\sf,1.25*\sf);
\path
    (1) edge (2)
    	edge (3)
    (2) edge (3)
    (3) edge (6)
    (4) edge (6)
        edge (5)
    (5) edge (6);
\end{tikzpicture}
    \caption{A cobipartite graph that is not saturated with respect to the partition shown.  For this graph, every zero forcing set must include vertices on both sides of the partition.}
\label{fig:first unsaturated example}
\end{figure}
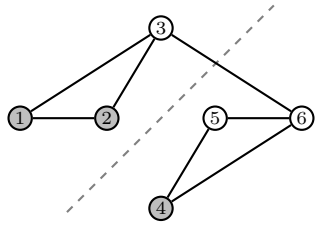

\begin{thm}\th\label{thm:triangle number versus ZF for cobipartite}
Let $\Y$ be an $m\times n$ zero-nonzero pattern and let $G$ be
the cobipartite graph associated with $\Y$.
If $G$ is saturated with respect to the corresponding clique partition, then
\begin{equation}
    Z(G) = m+n-\tri(\Y).    
\end{equation}
\end{thm}

\begin{proof}
Let $(V_1,V_2)$ be the clique partition of $G$ with vertices in $V_1$ corresponding to the rows of $\Y$ and those in $V_2$ corresponding to the columns.  Since $G$ is saturated with respect to this partition, \th\ref{normalized forcing for cobipartite} gives an optimal zero forcing set $F$ for $G$ that forces from $V_1$ to $V_2$.  Hence, 
\th\ref{cor:min ZF set from left to right}
gives the desired equality.
\end{proof}

Note that it is also possible to obtain \th\ref{thm:triangle number versus ZF for cobipartite} by combining Theorems 20, 44, and 45 from \cite{jephian_masters_thesis}.

\begin{ex}\th\label{ex:cobipartite with bounds}
\begin{figure}[ht]
    \centering
\begin{minipage}{2in}
\begin{tikzpicture}[thick, main node/.style={circle,minimum width=7pt,draw,inner sep=1pt}, filled node/.style={fill={lightgray},circle,minimum width=7pt,draw,inner sep=1pt}]
\def \sf {1.2}
\draw
    (0*\sf,1*\sf) node[main node] (3) {\scriptsize 3}
    (0*\sf,-1*\sf) node[main node] (4) {\scriptsize 4}
    (0.6*\sf,0*\sf) node[main node] (5) {\scriptsize 5}
    (1.56*\sf,0*\sf) node[main node] (6) {\scriptsize 6}
    (-0.6*\sf,0*\sf) node[main node] (2) {\scriptsize 2}
    (-1.56*\sf,0*\sf) node[main node] (1) {\scriptsize 1}
    ;
\draw[color=gray,dashed] (-1.25,-1.25) -- (1.25*\sf,1.25*\sf);
\path
    (1) edge (2)
    	edge (3)
		edge (4)
    (2) edge (3)
    	edge (4)
        edge (5)
    (3) edge (4)
        edge (5)
        edge (6)
    (4) edge (6)
        edge (5)
    (5) edge (6);
\end{tikzpicture}
\end{minipage}
\hspace{0.25in}
\begin{minipage}{2in}
       $\Y = \kbordermatrix{
        ~ & 4 & 5 & 6 \\
        1 & * & 0 & 0 \\
        2 & * & * & 0 \\
        3 & * & * & * \\
        }$
\end{minipage}
    \caption{The graph $G$ shown above is the cobipartite graph associated with the pattern $\Y$ given, with the dashed line indicating the corresponding partition of $G$ into cliques.  Note that $G$ is saturated with respect to this partition.}
    \th\label{fig:3-tree on 6 vertices}
\end{figure}
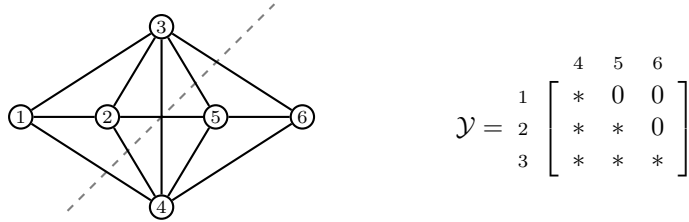

Let $G$ be the graph shown in Figure \ref{fig:3-tree on 6 vertices}.  Then $G$ is the cobipartite graph associated with the pattern $\Y$ given there, and is saturated with respect to this partition.  Since $\Y$ is itself a triangle, $\tri(\Y)=3$, and so from \th\ref{thm:triangle number versus ZF for cobipartite} we have $Z(G)=3+3-3=3$.
\end{ex}

When the cobipartite graph $G$ associated with the zero-nonzero pattern $\Y$ is saturated with respect to the corresponding partition,
\thref{thm:triangle number versus ZF for cobipartite}
gives a relationship between the parameter $Z(G)$ that bounds the maximum nullity $M(G)$ on the one hand and
the parameter $\tri(\Y)$ that bounds the minimum rank $\mr(\Y)$ on the other.
It was previously shown in \cite{not_necessarily_symmetric} that the corresponding quantities that are bounded, namely $M(G)$ and $\mr(\Y)$, obey exactly the same relationship, at least when considered over an infinite field.

\begin{thm}[{\cite[Theorem 3.1]{not_necessarily_symmetric}}]\th\label{thm:not necessarily symmetric min rank result}
Let $G$ be the cobipartite graph associated with the $m\times n$ zero-nonzero pattern $\Y$, and suppose $G$ is saturated with respect to the corresponding partition.  Then, over an infinite field, $\mr(\Y)=\mr(G)$, or, equivalently,
\begin{equation}\label{eqn:max nullity and pattern min rank}
    M(G) = m+n-\mr(\Y).    
\end{equation}
In addition, over the reals, $\mr(\Y)=\mr_+(G)$, or, equivalently,
\[
    M_+(G) = m+n-\mr(\Y).
\]
\end{thm}

Combining \th\ref{thm:triangle number versus ZF for cobipartite,thm:not necessarily symmetric min rank result} gives the following result
that, for a pattern $\Y$,
relates the question of whether $\mr(\Y)=\tri(\Y)$ to that of whether $M(G)=Z(G)$ for the associated cobipartite graph $G$.

\begin{thm}\th\label{thm:equality triangle number and equality Z}
Let $G$ be the cobipartite graph associated with the $m\times n$ zero-nonzero pattern $\Y$, and suppose $G$ is saturated with respect to the corresponding partition.
Then, over any field, $M(G)=Z(G)$ implies $\mr(\Y)=\tri(\Y)$.  Moreover, when the field is infinite, the converse also holds.
\end{thm}

\begin{proof}
By \th\ref{thm:triangle number versus ZF for cobipartite},
the statement that $M(G)=Z(G)$ is equivalent to
$M(G)=m+n-\tri(\Y)$.  Since $G$ has $m+n$ vertices, this is equivalent to
$m+n-\mr(G)=m+n-\tri(\Y)$, which is clearly equivalent to
$\mr(G)=\tri(\Y)$.
Over an infinite field, \thref{thm:not necessarily symmetric min rank result} gives that this is equivalent to $\mr(\Y)=\tri(\Y)$.

Now, without any restriction on the field, assume $M(G)=Z(G)$.  Then, by the above,
$\mr(G)=\tri(\Y)$.
Note that every matrix in $\gqual{G}$ has a submatrix with the zero-nonzero pattern $\Y$, which implies that $\mr(G) \ge \mr(\Y) \ge \tri(\Y)$.  So here we have equality throughout.  In particular, then, $\mr(\Y)=\tri(\Y)$.
\end{proof}

\begin{ex}\th\label{ex:15 vx graph from 8x7 pattern}

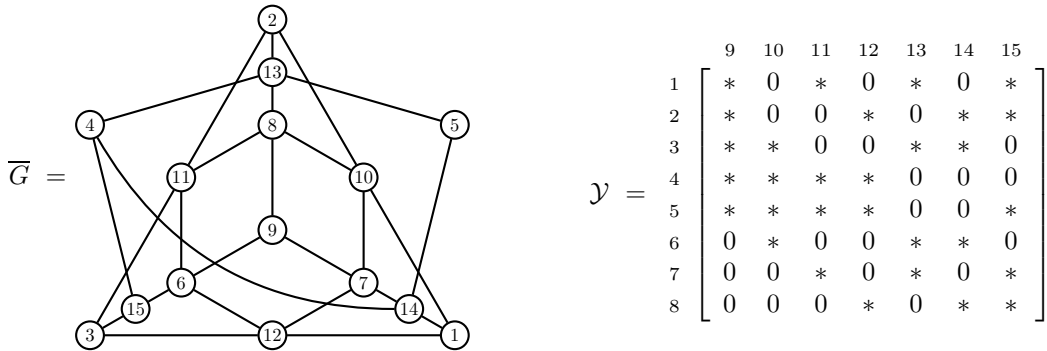
\begin{figure}[ht]
    \centering
$\overline{G} ~=~ $\begin{minipage}{2.1in}
\begin{tikzpicture}[thick, main node/.style={circle,minimum width=15pt,draw,inner sep=1pt,scale=0.7}, filled node/.style={fill={lightgray},circle,minimum width=7pt,draw,inner sep=1pt},scale=0.7]

\draw (0,4) node[main node] (2) {2};
\draw (0,3) node[main node] (13) {13};
\draw (0,2) node[main node] (8) {8};
\draw (0,0) node[main node] (9) {9};
\draw (2*0.866,2*0.5) node[main node] (10) {10};
\draw (2*-0.866,2*0.5) node[main node] (11) {11};
\draw (2*-0.866,2*-0.5) node[main node] (6) {6};
\draw (3*-0.866,3*-0.5) node[main node] (15) {15};
\draw (4*-0.866,4*-0.5) node[main node] (3) {3};
\draw (2*0.866,2*-0.5) node[main node] (7) {7};
\draw (3*0.866,3*-0.5) node[main node] (14) {14};
\draw (4*0.866,4*-0.5) node[main node] (1) {1};
\draw (2*0,2*-1) node[main node] (12) {12};
\draw (4*0.866,4*0.5) node[main node] (5) {5};
\draw (-4*0.866,4*0.5) node[main node] (4) {4};

\path
	(1) edge (10)
		edge (12)
		edge (14)
	(2) edge (10)
		edge (11)
		edge (13)
	(3) edge (11)
		edge (12)
		edge (15)
	(4) edge (13)
		edge[bend right=30] (14)
		edge (15)
	(5) edge (13)
		edge (14)
	(6) edge (9)
		edge (11)
		edge (12)
		edge (15)
	(7) edge (9)
		edge (10)
		edge (12)
		edge (14)
	(8) edge (9)
		edge (10)
		edge (11)
		edge (13)
;
\end{tikzpicture}

\end{minipage}
\hspace{0.5in}
\begin{minipage}{2.5in}
       $\Y ~=~ \kbordermatrix{%
    ~ & 9 & 10 & 11 & 12 & 13 & 14 & 15 \\   
	1 & * & 0 & * & 0 & * & 0 & * \\  %
	2 & * & 0 & 0 & * & 0 & * & * \\  %
	3 & * & * & 0 & 0 & * & * & 0 \\  %
	4 & * & * & * & * & 0 & 0 & 0 \\  %
	5 & * & * & * & * & 0 & 0 & * \\  %
	6 & 0 & * & 0 & 0 & * & * & 0 \\  %
	7 & 0 & 0 & * & 0 & * & 0 & * \\  %
	8 & 0 & 0 & 0 & * & 0 & * & * \\  %
}$
\end{minipage}
    \caption{The bipartite graph shown above is the complement of the cobipartite graph $G$ associated with the given zero-nonzero pattern $\Y$ (taken from \cite[Example 25]{deaett2020}) satisfying $\mr(\Y) > \tri(\Y)$ over every field.
    \thref{thm:equality triangle number and equality Z} therefore gives $M(G) < Z(G)$ over every field.
    }
    \label{fig:15 vertex graph with M neq Z}
\end{figure}

The $8\times 7$ zero-nonzero pattern $\Y$ given in Figure \ref{fig:15 vertex graph with M neq Z}
was shown in \cite[Example 25]{deaett2020} to have $\tri(\Y)=4 < \mr(\Y)$ over every field.  Hence, taking $G$ to be the $15$-vertex cobipartite graph associated with $\Y$  gives $M(G) < Z(G)$ over every field, by \th\ref{thm:equality triangle number and equality Z}.
(The complement of $G$ is shown in Figure \ref{fig:15 vertex graph with M neq Z}.)
\end{ex}

\begin{figure}[ht]
    \centering
\begin{tikzpicture}[thick, main node/.style={circle,minimum width=10pt,draw,inner sep=1pt,scale=0.7}, filled node/.style={fill={lightgray},circle,minimum width=7pt,draw,inner sep=1pt},scale=0.85]
    \foreach \x [count=\i] in {6,7,8,9,10}
    {
    \draw (\i*360/5+90:1) node[main node] (ins\i) {}; %
    \draw  (\i*360/5+90:2) node[main node] (out\i) {}; %
    \path (ins\i) edge (out\i);
    }
    \draw (ins1) edge (ins2)
          (ins2) edge (ins3)
          (ins3) edge (ins4)
          (ins4) edge (ins5)
          (ins5) edge (ins1);
      \end{tikzpicture}
    \hspace{0.75in}
\begin{tikzpicture}[thick, main node/.style={circle,minimum width=10pt,draw,inner sep=1pt,scale=0.7}, filled node/.style={fill={lightgray},circle,minimum width=7pt,draw,inner sep=1pt},scale=0.85]
    \foreach \x [count=\i] in {6,7,8,9,10}
    {
    \draw (\i*360/5+90:1) node[main node] (ins\i) {}; %
    \draw  (\i*360/5+90:1.75) node[main node] (out\i) {}; %
    \draw  (\i*360/5+90:2.5) node[main node] (wayout\i) {}; %
    \path (ins\i) edge (out\i);
    \path (out\i) edge (wayout\i);
    }
    \draw (ins1) edge (ins2)
          (ins2) edge (ins3)
          (ins3) edge (ins4)
          (ins4) edge (ins5)
          (ins5) edge (ins1);
      \end{tikzpicture}
    
    \caption{The pentasun graph, also called the $5$-sun or $5$-corona (left) and the extended pentasun graph (right)}
    \label{fig:pentasun}
\end{figure}
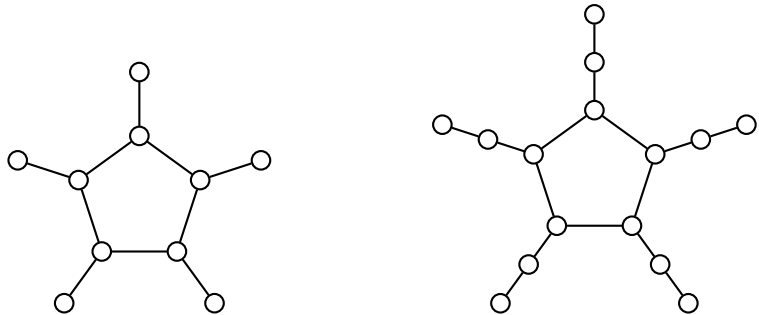

Letting $P$ be the pentasun graph shown in Figure \ref{fig:pentasun}, it is proved in
\cite[Example 4.1]{original_ZF_paper} that
$M(P) < Z(P)$ over every field, the same gap exhibited by the graph $G$ of \thref{ex:15 vx graph from 8x7 pattern}.  But \cite[Example 2.25]{tree-width-JGT} shows that in fact $\hat Z(P) < Z(P)$, which in some sense explains the gap, since $M(P) \le \hat Z(P)$ over every field.  It is then natural to ask whether or not this same situation explains the gap for the graph $G$.  Hence, we turn our attention to $\hat Z (G)$ and other zero forcing parameters for cobipartite graphs.

\subsection{Zero forcing parameters for cobipartite graphs}

Because \thref{treewidthdiagram} gives
\[
Z_+(G) \le \hat Z(G) \le Z(G),
\]
we will gain information about all three of these parameters by proving a strong result about $Z_+(G)$.  For this, we first need the two following lemmas.

\begin{lem}\th\label{lem:first_force_is_psd}
Let $G$ be a cobipartite graph and let $F$ be a positive semidefinite zero forcing set for $G$.  Then, in any fixed forcing sequence from $F$, any force that is a standard force is followed only by standard forces.  That is, all of the strictly positive semidefinite forces occur before any of the standard forces.
\end{lem}

\begin{proof}
Let $(A,B)$ be a  clique partition of $G$.  Fix any forcing sequence from $F$ and
consider the first standard force in this sequence.  By symmetry, we may assume that the vertex that forces is $a \in A$.  After this force occurs, every neighbor of $a$ is filled; in particular, every vertex in $A$ is filled.  
It is not hard to see that therefore no subsequent force can be a 
strictly positive semidefinite force.
\end{proof}

\begin{lem}\th\label{lem:two_unfilled_vertices}
Let $G$ be a cobipartite graph with clique partition $(A,B)$.
Let $F$ be a positive semidefinite zero forcing set for $G$
and suppose $|F|=Z_+(G) < Z(G)$.
If, in some complete forcing sequence from $F$, the first vertex to force is in $A$, then $A$ contains at least two initially unfilled vertices.
\end{lem}

\begin{proof}
Fix some complete forcing sequence from $F$, and let $a \in A$ be the first vertex to force in this sequence.
This first force must be a strictly positive semidefinite force since, otherwise,
by \th\ref{lem:first_force_is_psd}, 
all forces would be standard forces, contradicting $Z_+(G) < Z(G)$.
Hence, the initially unfilled subgraph of $G$ must have at least two connected components.  This implies that each of $A$ and $B$ contains at least one unfilled vertex, and that no unfilled vertex in $A$ is adjacent to any unfilled vertex in $B$.

Assume for the sake of contradiction that some vertex $y$ is the unique vertex in $A$ that is initially unfilled.  Then $y$ has no unfilled neighbors in $A$, and, by the above, $y$ has no unfilled neighbors in $B$ either.  Since $y$ has no unfilled neighbors, we can consider the set $F' = (F\setminus\{a\})\cup\{y\}$ as an initially-filled set.  A possible first force from $F'$ is then to let $y$ force $a$.  Following that force, the set of filled vertices is strictly a superset of $F$, showing that $F'$ is a zero forcing set. But since that first force from $F'$ is a standard force,
\th\ref{lem:first_force_is_psd} shows that $F'$ is actually a standard zero forcing set.  Since $|F'|=|F|=Z_+(G)$, this contradicts $Z_+(G) < Z(G)$.
\end{proof}

We are now prepared to prove that the ordinary zero forcing number and the positive semidefinite zero forcing number are equal for every cobipartite graph.  Note that this does not require any assumption that the cobipartite graph is saturated with respect to any partition.

\begin{thm}\th\label{thm:Z is Zplus for cobipartite}
If a graph $G$ is cobipartite, then
$Z_+(G)=Z(G)$.
\end{thm}

\begin{proof}
Let $G$ be a cobipartite graph with clique partition $(A,B)$.  
As \thref{treewidthdiagram} gives $Z_+(G) \le Z(G)$, it suffices to show $Z(G) \le Z_+(G)$.
So, for the sake of contradiction,
suppose there is a positive semidefinite zero forcing set $F$ for $G$ with $|F|=Z_+(G) < Z(G)$.

Fix a complete forcing sequence from $F$ and let $a_1 \to b_1$ be the first force in this sequence.  Without loss of generality, let $a_1 \in A$.  Then, by \thref{lem:two_unfilled_vertices}, at least two 
vertices of $A$ (and hence neighbors of $a_1$) are
initially unfilled.  (Note that this means they are not in $F$.)  It follows that the force is a strictly positive semidefinite force, and also that $b_1 \in B$.

By the above, there is some $k \ge 1$ such that $a_1\to b_1, a_2\to b_2, \ldots, a_k\to b_k$ is an initial subsequence of the forces with each $a_i \in A$, each $b_i \in B$, and each of the forces a strictly positive semidefinite force.  Crucially, then, we take $k$ to be as large as possible, meaning that both $F$ and the sequence of forces from $F$ are chosen so as to make the length of the maximal initial subsequence with the above properties as large as possible.
To complete the proof, we show that this choice leads to a contradiction, by modifying $F$ to obtain a still longer sequence with the same properties.

As noted above, at least two vertices of $A$ are initially unfilled.
Among those vertices, let $a'$ be the one that is the first to be forced, and let $b'$ be the vertex that forces it.  Then there are at least two unfilled vertices in $A$ (including $a'$) at the point when $b'$ is to force $a'$.  It follows that $b' \in B$.

Since the first force that occurs is a strictly positive semidefinite force, the set of initially unfilled vertices induces two connected components in the graph.  
Thus, the cobipartite structure of the graph implies that no initially unfilled vertex in $A$ is adjacent to any initially unfilled vertex in $B$.
In particular, every neighbor of $a'$ in $B$ is initially filled.
That is, $N(a') \cap B \subseteq F$.
In particular, $b' \in F$.

Now let $F'=(F\setminus \{b'\})\cup\{a'\}$.  
Then, by the above,
every neighbor of $a'$ in $B$ is in $F'$ except for $b'$.
Therefore, when $F'$ is taken as the initially filled set,
$a'$ begins with $b'$ as its only unfilled neighbor in $B$.
Moreover, since $b'\in F$ and $a' \not\in F$, we have $|F'|=|F|$.

As noted above, with $F$ as the set of filled vertices, the unfilled vertices induce two connected components.
We claim that this remains true if $F'$ is taken as the set of filled vertices.
Then the unfilled vertices are the same ones, except that they include $b'$ but not $a'$.
Hence, assuming for the sake of contradiction that these induce a single connected component, some edge adjacent to $b'$ must be a cut edge for the unfilled subgraph.  This must be an edge from $b' \in B$ to some vertex $a'' \in A$ not in $F'$.  
Since $a'\in F'$, we have $a''\neq a'$.  
In addition, since every vertex of $A$ that is in $F$ is also in $F'$, the fact that $a'' \not\in F'$ implies that $a''\not\in F$.
But now consider the original sequence of forces from $F$.
Because $a'$ was chosen as the first vertex of $A$ to be forced, $a''$ must remain unfilled at the point when the force $b'\to a'$ is to occur.  But both are neighbors of $b'$ and, since $A$ is a clique, they are in the same connected component of the graph induced by the unfilled vertices at that point.  But this implies that the force is impossible, and that contradiction proves our claim.  Hence, the vertices in $V(G)\setminus F'$ induce two connected components.

We now argue as follows.  Take $F'$ to be the initially filled set.  Then, since at least two vertices of $A$ are not in $F$, the choice of $F'$ implies that at least one vertex of $A$ is unfilled.  Let $\hat a \in A$ be such a vertex.  As noted above, before any forces occur, $a'$ is filled and has a single unfilled neighbor in $B$, namely $b'$.  Since the unfilled vertices induce two connected components, any unfilled neighbors of $a'$ in $A$ do not affect whether the positive semidefinite forcing rule allows $a'$ to force $b'$, so this can be taken as the initial force.
After this force occurs, the set of filled vertices is $F'\cup\{b'\}$, which is a superset of the positive semidefinite zero forcing set $F$.  It follows that we may take
\begin{equation}\label{eqn:longer initial subsequence of forces}
    a'\to b', a_1 \to b_1,a_2\to b_2,\ldots, a_k\to b_k
\end{equation}
as the initial subsequence of a sequence of forces from $F'$ that results in every vertex being filled.
Because $b'\in F$ while $b_i\not\in F$ for each $i$, the $k+1$ forces in \eqref{eqn:longer initial subsequence of forces} are distinct and form a valid sequence of forces; in particular, the new initial force does not obviate any of the later ones.
Since $|F'|=|F|=Z_+(G)$, the set $F'$ is an optimal positive semidefinite zero forcing set for $G$.
Moreover, each force in \eqref{eqn:longer initial subsequence of forces} features a vertex in $A$ forcing a vertex in $B$.  And, since the vertex $\hat a$ remains unfilled throughout this sequence of forces, each one is a strictly positive semidefinite force.
So now the fact that \eqref{eqn:longer initial subsequence of forces} is a sequence of more than $k$ such forces is a contradiction to the choice of
$F$ and $k$.
\end{proof}

\begin{cor}\th\label{cor:equality of four params for cobipartite}
If a graph $G$ is cobipartite, then $Z_+(G) =\hat Z(G) = Z(G)$.
\end{cor}

\begin{proof}
By \th\ref{treewidthdiagram}, we have $Z_+(G) \le \hat Z(G) \le Z(G)$.    
\th\ref{thm:Z is Zplus for cobipartite} then gives equality throughout.
\end{proof}

\begin{ex}\label{ex:15 vx graph inequality with Zhat}
As noted in \thref{ex:15 vx graph from 8x7 pattern}, 
the graph $G$ on $15$ vertices given in Figure \ref{fig:15 vertex graph with M neq Z} has $M(G) < Z(G)$ over every field.  \thref{cor:equality of four params for cobipartite} now gives the stronger inequality of $M(G) < \hat Z(G) = Z(G)$ over every field.  That is, the maximum nullity $M(G)$ is not only strictly less than the zero forcing number $Z(G)$, but is in fact less than the enhanced zero forcing number $\hat Z(G)$, and this gap persists over every field.

The smallest graph already known to display such a gap
is the extended pentasun graph shown in Figure \ref{fig:pentasun},
also with $15$ vertices.  Letting $P$ denote this graph, $M(P)=2$ holds over every field, as can be calculated using, e.g., \cite[Theorem 2.3]{BFH2004}, while \cite[Proposition 6.1]{jephian2016} gives $\hat Z(P)=Z(P)=3$.

Another related example is the Heawood graph, which we denote here by $H$.  
It happens that $H$ is bipartite, so that its complement $\overline H$ is cobipartite.
In fact, \cite[Observation A.21]{tree-width-JGT} gives $M(\overline H) = 10$, while $\hat Z(\overline H)=Z(\overline H)=11$, 
so that $\overline H$ displays the same gap observed above.
Moreover, $\overline H$ has only $14$ vertices.  But here the gap relies on properties of the real field; the discussion in \cite[Example 37]{deaett2020} together with \thref{thm:not necessarily symmetric min rank result} show that in fact
$M(\overline H)=11=\hat Z(\overline H)=Z(\overline H)$ over every infinite field of characteristic $2$.
\end{ex}

The following theorem extends the equivalence given by \thref{thm:equality triangle number and equality Z} to the enhanced zero forcing number, and, over the real numbers, to the positive semidefinite setting as well.

\begin{thm}\th\label{thm:bounds_equality_equivalence}
Let $\Y$ be an $m\times n$ zero-nonzero pattern and let $G$ be
the cobipartite graph associated with $\Y$.
Suppose $G$ is saturated with respect to the corresponding clique partition.
Then, over an infinite field, the following are equivalent.
\begin{enumerate}
    \item\th\label{cond:mr equals tri}
    $\mr(\Y)=\tri(\Y)$
    \item\th\label{cond:M equals Z}
    $M(G)=Z(G)$
    \item\th\label{cond:M equals Zhat}
    $M(G)=\hat Z(G)$
\end{enumerate}
In addition, over $\mathbb R$, the above statements are equivalent to $M_+(G)=Z_+(G)$.
\end{thm}

\begin{proof}
\th\ref{thm:equality triangle number and equality Z} states that \ref{cond:mr equals tri} and \ref{cond:M equals Z} are equivalent.
In addition, $Z(G) = \hat Z(G)$ by \th\ref{cor:equality of four params for cobipartite}, and hence \ref{cond:M equals Z} is equivalent to \ref{cond:M equals Zhat} as well.

Turning to the situation over $\mathbb R$, we have by \th\ref{thm:Z is Zplus for cobipartite} that $Z(G)=Z_+(G)$, and hence the
statement that $M_+(G)=Z_+(G)$ is equivalent to the statement that $M_+(G)=Z(G)$.  Since it follows from the second part of \th\ref{thm:not necessarily symmetric min rank result} that $M_+(G)=M(G)$, this in turn is equivalent to $M(G)=Z(G)$, which is \ref{cond:M equals Z}.
\end{proof}

\begin{figure}[ht]
    \centering
\begin{minipage}{2.1in}
\begin{tikzpicture}[thick, main node/.style={circle,minimum width=7pt,draw,inner sep=1pt}, filled node/.style={fill={lightgray},circle,minimum width=7pt,draw,inner sep=1pt}]
\def \sf {1.2}
\draw
    (0*\sf,1*\sf) node[main node] (3) {\scriptsize 3}
    (0*\sf,-1*\sf) node[main node] (4) {\scriptsize 4}
    (0.6*\sf,0*\sf) node[main node] (5) {\scriptsize 5}
    (1.56*\sf,0*\sf) node[main node] (6) {\scriptsize 6}
    (2.52*\sf,0*\sf) node[main node] (7) {\scriptsize 7}
    (-0.6*\sf,0*\sf) node[main node] (2) {\scriptsize 2}
    (-1.56*\sf,0*\sf) node[main node] (1) {\scriptsize 1}
    ;
\draw[color=gray,dashed] (-1.25,-1.25) -- (1.25*\sf,1.25*\sf);
\path
    (1) edge (2)
    	edge (3)
		edge (4)
    (2) edge (3)
    	edge (4)
        edge (5)
    (3) edge (4)
        edge (5)
        edge (6)
    (4) edge (6)
        edge (5)
    (5) edge (6)
    (7) edge (6)
        edge (4)
        edge[bend right=45] (5)
        ;
\end{tikzpicture}
\end{minipage}
\hspace{0.75in}
\begin{minipage}{2in}
       $\Y = \kbordermatrix{
        ~ & 4 & 5 & 6 & 7 \\
        1 & * & 0 & 0 & 0 \\
        2 & * & * & 0 & 0 \\
        3 & * & * & * & 0 \\
        }$
\end{minipage}
    \caption{The graph $G$ shown is the cobipartite graph associated with the pattern $\Y$, but is not saturated with respect to the corresponding clique partition (due to the column in $\Y$ with only $0$ entries).}
    \label{fig:unsaturated-graph}
\end{figure}
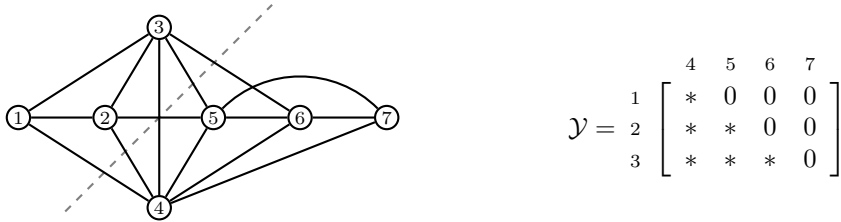

\begin{ex}\label{ex:bounds not saturated example}
Let $G$ be the graph shown in Figure \ref{fig:unsaturated-graph}. Then $G$ is the cobipartite graph associated with the pattern $\Y$ given there, but is not saturated with respect to the corresponding clique partition.
Since $G$ has a clique of size $4$, we have $h(G) \ge 4$.  Hence, we may apply
\th\ref{treewidthdiagram} to obtain
\[
3 \le h(G)-1 \le M(G) \le \hat Z(G) \le Z(G) \le 3,
\]
where the final inequality follows from verifying that $\{1,2,3\}$ is a zero forcing set.
In particular, then, $M(G)=Z(G)=3$, and hence $\mr(G)=7-M(G)=4$. Meanwhile, it is clear that $\mr(\Y)=\tri(\Y)=3$.
Hence, $\mr(\Y) < \mr(G)$, showing that
the condition that $G$ be saturated is essential to \th\ref{thm:not necessarily symmetric min rank result}.
\end{ex}

\subsection{Treewidth and cobipartite \texorpdfstring{\(k\)}{k}-trees}

Of the parameters mentioned in \thref{treewidthdiagram}, the only one not yet treated here is $\tw(G)$.  In particular, while the theorem gives $\tw(G) \leq Z_+(G) \le \hat Z(G) \le Z(G)$ for every graph $G$, we saw in \thref{cor:equality of four params for cobipartite} that when $G$ is cobipartite, $Z_+(G) = \hat Z(G) = Z(G)$.  Hence, it is natural to ask whether this equality can be extended to include $\tw(G)$.  We will show that this question has a negative answer for cobipartite graphs in general, but a positive one when $G$ happens also to be a $k$-tree, defined as follows.

\begin{defn}\th\label{def:k-tree}
For each integer $k \ge 1$, a graph is said to be a \defterm{$k$-tree} if it
can be constructed by starting with a $(k+1)$-clique and iteratively repeating
this step (at least zero times): Choose a set of $k$ vertices that induce a
clique and add a new vertex that is adjacent to exactly those $k$ vertices. A
\emph{linear  k-tree} is a \( k \)-tree that either is a \((k+1)\)-clique, or has exactly two vertices of degree
\( k \).
\end{defn}

\begin{ex}
The graph $G$ of \th\ref{ex:cobipartite with bounds}, shown in Figure \ref{fig:3-tree on 6 vertices}, is a $3$-tree, as it can be formed by starting with a clique with vertex set $\{1,2,3,4\}$, then adding vertex $5$ with neighborhood $\{2,3,4\}$, and finally adding vertex $6$ with neighborhood $\{3,4,5\}$.  Since vertices \( 1 \) and \( 6 \) are the only two vertices of degree \( 3 \), by definition \( G \) is a linear \( k \)-tree.
\end{ex}

\begin{defn}\th\label{def:treewidth}
Let $G$ be a graph.  Then the \defterm{treewidth} of $G$, denoted $\tw(G)$, is the smallest integer $k$ such that $G$ is a subgraph of some $k$-tree.
\end{defn}

\th\ref{cor:equality of four params for cobipartite} shows that the latter four parameters appearing in the inequality \eqref{eqn:parameters from diagram inequality} are equal whenever $G$ is cobipartite.  The next example shows that this result cannot be extended to include the treewidth.

\begin{ex}\label{twnotZ}
Let $H$ be the graph shown in Figure \ref{fig:inequality_with_tw}.  Then $H$ is
obtained from the graph $G$
shown in Figure \ref{fig:3-tree on 6 vertices}
by deleting the edge between $2$ and $5$.
Since $G$ is a $3$-tree, this shows that $\tw(H) \le 3$.  Since $\delta(H)=3$, we have by \thref{treewidthdiagram} that $\delta(H)=\tw(H)=3$.  Note that $H$ is still cobipartite, and saturated with respect to the same partition as $G$.  At the same time, 
removing that edge from $G$ decreases the triangle number of the corresponding zero-nonzero pattern to $2$, so that \th\ref{thm:triangle number versus ZF for cobipartite} gives $Z(H)=3+3-2=4$.
We can now apply \th\ref{cor:equality of four params for cobipartite} to obtain
\[
\delta(H) = \tw(H) = 3 < 4 = \hat Z(H) = Z_+(H) = Z(H).
\]

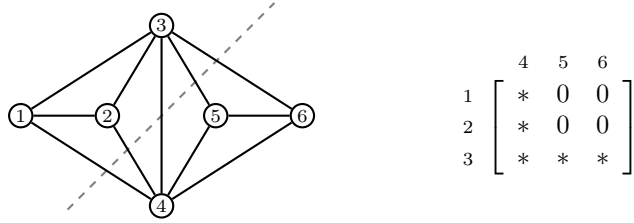
\begin{figure}
    \centering
\begin{minipage}{2in}
\begin{tikzpicture}[thick, main node/.style={circle,minimum width=7pt,draw,inner sep=1pt}, filled node/.style={fill={lightgray},circle,minimum width=7pt,draw,inner sep=1pt}]
\def \sf {1.2}
\draw
    (0*\sf,1*\sf) node[main node] (3) {\scriptsize 3}
    (0*\sf,-1*\sf) node[main node] (4) {\scriptsize 4}
    (0.6*\sf,0*\sf) node[main node] (5) {\scriptsize 5}
    (1.56*\sf,0*\sf) node[main node] (6) {\scriptsize 6}
    (-0.6*\sf,0*\sf) node[main node] (2) {\scriptsize 2}
    (-1.56*\sf,0*\sf) node[main node] (1) {\scriptsize 1}
    ;
\draw[color=gray,dashed] (-1.25,-1.25) -- (1.25*\sf,1.25*\sf);
\path
    (1) edge (2)
    	edge (3)
		edge (4)
    (2) edge (3)
    	edge (4)
    (3) edge (4)
        edge (5)
        edge (6)
    (4) edge (6)
        edge (5)
    (5) edge (6);
\end{tikzpicture}
\end{minipage}
\hspace{0.25in}
\begin{minipage}{2in}
       $\kbordermatrix{
        ~ & 4 & 5 & 6 \\
        1 & * & 0 & 0 \\
        2 & * & 0 & 0 \\
        3 & * & * & * \\
        }$
\end{minipage}
    \caption{The graph $H$ above is cobipartite  with $\tw(G) < Z(G)$.  It is the cobipartite graph associated with the zero-nonzero pattern shown, and is saturated with respect to the corresponding clique partition.}
    \th\label{fig:inequality_with_tw}
\end{figure}
\end{ex}

On the other hand, we will show in \th\ref{equality with treewidth for cobipartite k-trees} that
equality among all of the parameters appearing in \eqref{eqn:parameters from diagram inequality} in fact must occur for a cobipartite graph that is actually a $k$-tree.  For convenience in what follows, for the remainder of the section we let $k$ represent an arbitrary positive integer.

\begin{thm}[{\cite[Lemma 2.37]{tree-width-JGT}}]\th\label{z_linear_k_tree}
Let \( G \) be a linear \( k \)-tree. Then \( Z(G) = k \).
\end{thm}

\begin{thm}\th\label{equality with treewidth for linear k-trees}
If $G$ is a linear $k$-tree, then
\(
    h(G)-1 = M(G) =  Z(G) = k.
\)
\end{thm}

\begin{proof}
Let $G$ be a linear $k$-tree.  Then, by definition, \( G \) contains a \( (k + 1)
\)-clique as a subgraph, so that \(  k + 1 \leq h(G)
\). Hence, using \th\ref{treewidthdiagram}, we have
\[
k \le h(G)-1 \le  M(G) \leq Z(G) = k,
\]
where the final equality follows from \th\ref{z_linear_k_tree}.
\end{proof}

The lemmas that follow will allow us to exploit the above results in the cobipartite case specifically.

\begin{lem}\thlabel{lem:k-tree no two vxs of degree k}
In a \(k\)-tree that is not a clique, no two vertices of degree \(k\) can be adjacent.
\end{lem}

\begin{proof}
Toward a contradiction, let \(G\) be a \(k\)-tree that is not a clique in which some two vertices of degree \(k\) are adjacent.  Since \(G\) is not a clique,
we may assume without loss of generality that in the construction of \(G\) as a \(k\)-tree, one of these was the last vertex to be added.  Call this vertex \(w\).  Let \(v\) be some other supposed vertex of degree \(k\) that is adjacent to \(w\).
Then \(G\setminus w\) is a \(k\)-tree in which \(v\) has degree \(k-1\).  But this is a contradiction, since every vertex of a \(k\)-tree has degree at least \(k\).
\end{proof}

\begin{lem}\thlabel{lem:k-tree at least two vxs degree k}
Every \(k\)-tree has at least two vertices of degree \(k\).
\end{lem}

\begin{proof}
Suppose this is false; let \(H\) be a minimal counterexample.
Since \(H\) is a \(k\)-tree, it must contain at least one vertex of degree \(k\).
Since \(H\) is a counterexample, it contains exactly one such vertex.  Clearly, then, \(H\) is not a clique.  Hence, we may consider the last vertex that was added in its construction as a \(k\)-clique; call this vertex \(z\).  
Then \(H\setminus z\) is a \(k\)-tree that,
by the minimality of \(H\), must have at least two vertices of degree \(k\).  Let \(x\) and \(y\) be two of these.  Since \(H\) is a counterexample, these no longer have degree \(k\) as vertices in \(H\), and so it must be that \(z\) is adjacent to both of them.  Since the neighborhood of \(z\) forms a clique, it follows that \(x\) and \(y\) are adjacent in \(H\setminus z\).  But \(H\setminus z\) is a \(k\)-tree, so this contradicts \thref{lem:k-tree no two vxs of degree k}.
\end{proof}

\begin{thm}\thlabel{thm:cobipartite k-tree is linear k-tree}
Every cobipartite \(k\)-tree is a linear \(k\)-tree.
\end{thm}

\begin{proof}
Let \(G\) be a \(k\)-tree that is cobipartite, say with clique partition \( (V_1, V_2) \). If \( G \) is itself a clique, then \( G \) is a linear \( k
\)-tree by definition. Hence, suppose \( G \) is not a clique. Then it suffices to show that
\(G\) has exactly two vertices of degree \(k\). We have by
\th\ref{lem:k-tree at least two vxs degree k} that \( G \) contains at least two
vertices of degree \( k \).
Since \(V_1\) and \(V_2\) are both cliques, it follows from
\th\ref{lem:k-tree no two vxs of degree k} that each contains at most one vertex of degree \( k \). Therefore, \( V_1 \) and \( V_2 \) each contains exactly one
vertex of degree \( k \).  So \(G\) contains exactly two such vertices, as desired.
\end{proof}

While Example \ref{twnotZ} shows that
a cobipartite graph $G$ may exhibit a gap such that $\tw(G) < Z(G)$,
combining \thref{thm:cobipartite k-tree is linear k-tree,equality with treewidth for linear k-trees}
shows that this cannot occur for a cobipartite graph that is a $k$-tree.

\begin{cor}\th\label{equality with treewidth for cobipartite k-trees}
Let $G$ be a cobipartite $k$-tree.  Then
\(
    h(G) - 1 =  M(G) = Z(G) = k.
\)
\end{cor}

For a general graph $G$, almost all of the combinatorial parameters included in \cite[Figure
1.1]{tree-width-JGT} are shown there to be between a lower bound of $h(G)-1$ and an upper bound of $Z(G)$.  Since any graph with a clique of size \( Z(G) +1 \) satisfies $Z(G) \le h(G)-1$,  
all of those parameters are equal, and in particular
\[
h(G) - 1 = \tw(G) = M_+(G) = M(G) = Z_+(G) = \hat Z(G) = Z(G)
\]
for every such graph $G$.

\section{Unsaturated cobipartite graphs}\label{sec:unsaturated}

Many of our results so far have focused on the case of a cobipartite graph that is saturated
with respect to some partition of its vertices into two cliques.
When a cobipartite graph
does not satisfy this condition,
we refer to it as an \emph{unsaturated cobipartite graph}.
In this section, we turn our attention to such graphs.

The condition of being saturated means (\thref{def:saturated}) that each vertex in each clique is adjacent to some vertex in the
other clique.
Hence, to understand the zero forcing number of unsaturated cobipartite graphs, we need to account for the way in which the zero forcing number is
affected by the presence of vertices that violate this condition. Since a
saturated cobipartite graph can always be obtained by deleting all such
vertices, we examine what effect this deletion has
on the zero forcing number.  For this, we wish to employ some useful results
from \cite{MR2917414}; the following definitions make it easier to do so.

\begin{defn}\thlabel{def:source-and-terminal}
Let $k \ge 1$ be an integer and let $(v_0,v_1,\ldots,v_k)$ be a forcing chain.  Then $v_0$ is the \defterm{source} vertex of the forcing chain and $v_k$ is its \defterm{terminal} vertex.
\end{defn}

Note that the definition of source vertex and terminal vertex requires that the chain consist of more than one vertex; by definition, a forcing chain of length $0$ has neither a source vertex nor a terminal vertex.

\begin{defn}\label{defn:critical and uncritical}
A vertex is said to be \defterm{critical} if there exists some optimal forcing sequence in which it does not force or get forced.  It is said to be \defterm{uncritical} if, for all optimal forcing sequences,
it is either the terminal or the source vertex of the chain in which it resides.
\end{defn}

To say that a vertex $v$ is critical is the same as saying that there exists a set of forcing chains, from some optimal zero forcing set $F$, such that the chain containing $v$ contains no other vertices. (Note how this implies that $v \in F$.)
According to \th\ref{def:source-and-terminal}, in this situation $v$ is considered to be neither a source vertex nor a terminal vertex, and hence is not uncritical.
Therefore, a critical vertex cannot be uncritical.  (A vertex can, on the other hand, be neither critical nor uncritical.)
Meanwhile, considering reversals shows that an uncritical vertex must be both a terminal vertex for some optimal forcing sequence and a source vertex for some other such sequence.

The main role for the notions of critical and uncritical vertices
is in accounting for the effect that removing a vertex has on the zero forcing number.

\begin{thm}[{\cite[Theorem 2.7]{MR2917414}}]
\th\label{thm:critical}
A vertex $z$ of $G$
is critical if and only if 
\( Z(G \setminus z) =
Z(G) - 1 \).
\end{thm}

\begin{prop}
\th\label{prop:uncritical}
Let \( z \) be an uncritical vertex in \( G \). Then \( Z(G \setminus z) =
Z(G)\).
\end{prop}

\begin{proof}
Since \( z \) is an uncritical vertex, it is a terminal vertex with respect to a complete forcing sequence from
some optimal zero forcing set \( F \). Hence, \( F \) is also a zero forcing set
for \( G \setminus z\), where the forces occur as in \( G \). Thus, \(
Z(G) \geq Z(G \setminus z) \).  
Suppose for the sake of contradiction that
\( Z(G) > Z(G \setminus z) \). Let \( F' \) be an optimal zero
forcing set for \( G \setminus z \). It follows that \( F' \cup  \{ z \} \) is an optimal zero forcing set for
\( G \) from which there is some forcing sequence of which \( z \) is a critical vertex. But this is impossible, since \( z \) is an
uncritical vertex.
\end{proof}

Note that \cite[Example 2.9]{MR2917414} shows that the converse of \th\ref{prop:uncritical} is false.

The following theorems give us tools to identify when a vertex that is contained in a clique of a graph is critical or uncritical.

\begin{thm}
\th\label{thm:criticaloruncritical}
Let \( G \) be a graph, let \( T \) be a clique in \( G \), and let \( z \) be a vertex of \( T \)
with \( \deg(z) = |T| - 1 \). Then \( z \) is either a
critical vertex or an uncritical vertex, and it follows that \( Z(G \setminus z ) \in \{
Z(G), Z(G) - 1 \} \).
\end{thm}

\begin{proof}
Suppose \( z \) is not a critical vertex and not an uncritical vertex. Then
there exist vertices 
\(u\) and \(v\) such that \( u \) forces \( z \) and \(
z \) forces \( v \) in some optimal forcing sequence. 
But since \( \deg(z) = |T| - 1 \), only vertices in \(T\) are adjacent to \(z\).  Hence, \(u\) and \(v\) are both in \(T\), and therefore are adjacent, which contradicts that \( u \) has exactly one unfilled neighbor when it forces
\( z \).
Hence, \(z\) must be either critical or uncritical, and the claim that \( Z(G -
z ) \in \{ Z(G), Z(G) - 1 \} \) follows from
\th\ref{thm:critical,prop:uncritical}.
\end{proof}

\begin{prop}
\th\label{cor:deg_critical}
Let \( G \) be a
graph, let \( T \) be a clique in \( G \), and let \( z \) be a vertex of \( T \) with $\deg(z)=|T|-1$.  Suppose that \( z \) is not the only vertex in
\( T \) of degree \(|T| - 1 \), that
some vertex in $T$ has a neighbor outside of $T$,
and that for each vertex in $T$, its set of neighbors outside of $T$ is either empty or forms a clique.
Then \( z \) is a
critical vertex.
\end{prop}

\begin{proof}
By hypothesis, there exists some $v\in T$ with $v\neq z$ and $\deg(v)=|T|-1$.
Suppose for the sake of contradiction that $z$ is not critical.
By \th\ref{thm:criticaloruncritical}, each of $z$ and $v$ must be critical or uncritical.
So $z$ is uncritical.

Case 1: $v$ is critical.  Since $z$ is uncritical, by definition $z$ is a terminal or source vertex for every optimal forcing sequence.  Since $v$ is critical, there exists some optimal forcing sequence in which $v$ neither forces nor gets forced; in particular, with this sequence, $v$ is not a terminal or source vertex.  But since $v$ and $z$ are adjacent twins, we can swap their roles in the forcing process to obtain an optimal forcing sequence such that $z$ is not a terminal or source vertex, contradicting the above.

Case 2: $v$ is uncritical.  Recall that so is $z$.  This means that for any optimal forcing sequence, each one is either a source or terminal vertex.
Moreover, since $v$ and $z$ are adjacent twins, we claim that they cannot both be source vertices.
To see this, suppose otherwise, and assume without loss of generality that $v$ forces before $z$.  The contradiction is that then $v$ would at that point have two unfilled neighbors, both the one that it is assumed to force, and also the one that $z$ is supposed to force later on.
By considering reversals, $v$ and $z$ cannot both be terminal vertices either.

Hence, one of $v$ and $z$ is source and the other is terminal.  Without loss of generality, say $v$ is terminal and $z$ is source.  Suppose first that $z$ forces some $y \neq v$.  Then, 
since $v$ is a neighbor of $z$, 
it must be that $v$ is forced by some vertex $x$ before this occurs.  By the choice of $z$ and $v$, we have $x,y\in T$, so that $x$ and $y$ are adjacent.  But then, at the point when $x$ is to force $v$, the vertices $v$ and $y$ are distinct unfilled neighbors of $x$, which is a contradiction.

Finally, suppose $z$ forces $v$.
By hypothesis, some vertex in $T$ has a neighbor outside of $T$. So there exists an $x\in T$ other than $z$ and $v$.  
If $v$ were the only initially unfilled vertex in $T$, then we could take $(F\setminus\{x\})\cup\{v\}$ as the initially filled set and let $v$ force $x$ as the initial force, with the subsequent forces proceeding as before.
But then $z$ would neither force nor get forced, and hence would be critical, contradicting that $z$ is uncritical.

Hence, some vertex $y\in T$ with $y\neq v$ is initially unfilled as well.  Since $y$ is a neighbor of $z$, it must be that $y$ is forced before the point at which $z$ is to force $v$, requiring that a vertex outside of $T$ forces $y$.
By assumption, the neighbors of $y$ outside of $T$ form a clique.  It follows that once $y$ is forced, so are all of its neighbors outside of $T$.  Therefore, at the point at which $z$ is to force $v$, we can instead let $y$ force $v$.  This again gives an optimal forcing sequence in which $z$ does not force or get forced, 
implying that $z$ is critical, and contradicting that $z$ is uncritical.
\end{proof}

\begin{prop}
\th\label{cor:deg_uncritical}
Let \( G \) be a graph, let \( T \) be a clique in \( G \), and let \( z \) be a vertex of \( T \)
with \( \deg(z) = |T| - 1 \ge 1 \).
Suppose that \( z \) is the only vertex in
\( T \) of degree \(|T| - 1 \)
and that \( \overline{T} \) is a clique.
Then \( z \) is an
uncritical vertex.
\end{prop}

\begin{proof}
Suppose \( \deg(z) = |T|-1 \) and no other vertex in \( T \) has this property.
\thref{thm:criticaloruncritical} gives that
\( z \) is either
critical or uncritical.
Suppose for the sake of contradiction that \( z \) is critical; from some optimal zero forcing set \( F \), fix an optimal forcing sequence in which \( z \) does not force or get forced.
As 
\(\deg(z) = |T| - 1 \ge 1\),
we have $|T| \ge 2$.

Case 1: \( T \subseteq F \). Since
$|T|\ge 2$, there exists \( x\in T \) with \( x \neq z \). Consider the set \( F \setminus \{ x \} \).
Since \( z \in T \), \( z \neq x \), and \( T \subseteq F \), it follows that \( z \in F
\setminus \{ x \} \). Since \( \deg(z) = |T| -1 \) and \( z \)
is in the clique \( T \), it follows
that \( N(z) \setminus \{ x \} \subseteq T \setminus \{ x \} \subseteq F \setminus \{ x \} \). Hence, with \( F \setminus \{ x \} \) as the initially filled set of vertices, \( z \to x \) can be the first force, after which the forcing can proceed as in the original sequence from \( F \).  But this gives a forcing sequence from \( F \setminus \{x\} \) that ends with all vertices filled, which contradicts that \( F \) is an optimal zero forcing set for \(
G \).

Case 2: \( T \not\subseteq F \).  Then some \( v \in T \) is initially unfilled. First suppose that there does not exist an optimal forcing sequence from \( F \) in which a vertex in \( \overline{T} \) can force \( v \). Hence, some \( x
\in T\) forces \( v \).
By our assumptions on \( z \), neither \( x \) nor \( v \) is equal to \( z \).  In particular, each of \( x \) and \( v \) has degree greater than \( |T| - 1 \), and so is adjacent to some vertex in \( \overline{T} \).  Hence, let \( a \in \overline{T} \) be adjacent to \( v \) and let \( b \in \overline{T} \) be adjacent to \( x \).

Since \( T \) is a clique, at the point when \( x
\) forces \( v \), it must be that \( T \setminus \{ v \} \) is filled.
At that point, by our assumption in this case, \( a \) cannot force \( v \).  Hence, either \( a \) is unfilled, or some neighbor \( w \neq v \) of \( a \) is unfilled.  In the latter case, \( w \in \overline{T} \), since every vertex in \( T \setminus \{ v \} \) is filled.  Either way, we have that \( \overline{T} \) contains some unfilled vertex.
It follows that no vertex in \( \overline{T} \) could have forced before \( x \) forced \( v \). Similarly, since \( T \) is a clique and
\( v \) is unfilled, no vertex in \( T \) could have forced either.
Thus, \( x \) is the first vertex to force, implying that \( b \in F \). 
Hence, let \( F' = (F \setminus \{ b \} ) \cup \{ v \}  \) and consider the sequence of forces from
\( F' \) where \( x \to b \) is the first force and the rest of the forces occur as in the forcing sequence from
\( F \). This shows that \( F' \) is an optimal zero forcing set such that \( T \subseteq F' \). By the argument of Case 1, this gives a contradiction.

Now suppose \( v \) gets forced by a vertex in \( \overline{T} \). Since \( T
\) is a clique and \( v \) is unfilled, all vertices that force before \( v \)
is forced must be in \( \overline{T} \). Hence, the forces that happen before \(
v\) gets forced are not dependent on \( z \) being filled. Since \( T \) is a
clique, if a vertex \( x \in T \) gets forced by a vertex in \( T \), without
loss of generality say \( v \) forces \( x \). Then \( x \) must be the last
vertex to get forced in \( G \) since \( \overline{T} \) is filled when \( x \)
got forced and \( T \) is a clique. By construction of \( G \), there exists \(
b \in \overline{T} \) such that \( x \) is adjacent to \( b \). Thus, \( b \)
can force \( x \). This shows there exists a complete forcing sequence from \( F \) in
which no vertex in \( T \) forces. Hence, all forces in the zero forcing
sequence are independent of \( z \) being filled. Thus, \( F \) can force \(
V(G)
\setminus \{ z \} \) without \( z \) being filled. Therefore, \( F \setminus \{
z \} \) can force \( V(G) \setminus \{ z \} \) and \( v \) can force \( z \) last.
This shows that \( F \setminus \{ z \} \) is a zero forcing set, which
contradicts that \( F \) is an optimal zero forcing set.
\end{proof}

\begin{ex}
\th\ref{cor:deg_uncritical} requires that the set of vertices outside of the clique $T$ must themselves form a clique (implying that the graph is cobipartite).  In \th\ref{cor:deg_critical}, the conditions on those vertices are weaker.  The graph in Figure \ref{fig:neighbors of T condition} shows that these weaker conditions would not be enough for \th\ref{cor:deg_uncritical}.
The clique $T$ indicated there satisfies the conditions of \th\ref{cor:deg_critical}.  In fact, it satisfies a stronger condition:\ the set of vertices outside $T$ that are adjacent to some vertex in $T$ itself induces a clique.
Nevertheless, vertex $1$ is critical (and therefore not uncritical), since 
there is an optimal zero forcing set (the vertices shown as filled in the figure) from which there is a complete sequence of forces such that vertex $1$ neither forces nor gets forced.

\begin{figure}
    \centering
\begin{tikzpicture}[thick, main node/.style={circle,minimum width=7pt,draw,inner sep=1pt}, filled node/.style={fill={lightgray},circle,minimum width=7pt,draw,inner sep=1pt}]
\def \sf {1.2}
\draw (-1.75*\sf,-1.25*\sf) node {$T$};
\draw
    (-1.56*\sf,0*\sf) node[filled node] (1) {\scriptsize 1}
    (0*\sf,1*\sf) node[main node] (2) {\scriptsize 2}
    (0*\sf,-1*\sf) node[main node] (3) {\scriptsize 3}
    (2*\sf,0.5*\sf) node[main node] (4) {\scriptsize 4}
    (2*\sf,-0.5*\sf) node[main node] (5) {\scriptsize 5}
    (3*\sf,1.25*\sf) node[filled node] (6) {\scriptsize 6}
    (3*\sf,-1.25*\sf) node[filled node] (7) {\scriptsize 7}
    (4*\sf,0.75*\sf) node[filled node] (8) {\scriptsize 8}
    (4*\sf,-0.75*\sf) node[main node] (9) {\scriptsize 9}
    ;
\draw[color=gray,dashed] (-2*\sf,-1.5*\sf) rectangle (0.5*\sf,1.5*\sf);
\path
    (1) edge (2)
    	edge (3)
    (2) edge (3)
    	edge (4)
    (3) edge (5)
    (4) edge (5)
    	edge (8)
    (5) edge (4)
	(6) edge (4)
		edge (8)
	(5) edge (7)
		edge (9);
\end{tikzpicture}
    \caption{Graph showing that \thref{cor:deg_uncritical} requires a stronger condition on the vertices outside of $T$ than does \thref{cor:deg_critical}.}
    \th\label{fig:neighbors of T condition}
\end{figure}
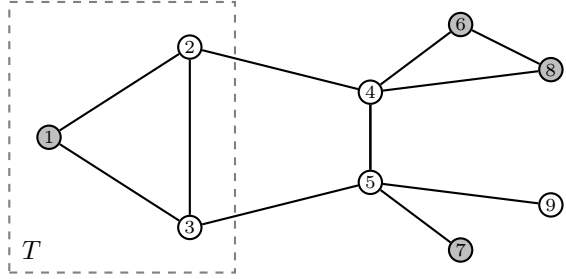
\end{ex}

\begin{defn}\th\label{unsaturated}
When a graph $G$ is not saturated with respect to a partition $(V_1,V_2)$ of
its vertex set, this is because some vertex in $V_1$ (or, respectively, $V_2$)
is not adjacent to any vertex in $V_2$ (respectively, $V_1$).  We say that
such a vertex is \defterm{unsaturated} with respect to that partition.  Where there is no risk of ambiguity and the partition can be understood from the context, we may simply refer to it as an \defterm{unsaturated vertex}.
\end{defn}

\begin{thm}
\thlabel{thm:Aunsaturated}
Let $G$ be a graph that is cobipartite with clique partition $(V_1,V_2)$
such that
\( V_1 \) has $k_1 \ge 1$ unsaturated vertices.  Let $H$ be the graph obtained by deleting all of the
unsaturated vertices in \( V_1 \). Then
\[
Z(G) = Z(H) + k_1 - 1.
\]
\end{thm}

\begin{proof}
Note that the unsaturated vertices of \(V_1\) are precisely those with  degree \( |V_1| - 1 \).
We may delete them one at a time.  
Each of these deletions that does not remove the last such vertex decreases the
zero forcing number by \(1\), by \th\ref{cor:deg_critical,thm:critical}.
Finally, when there is exactly one such vertex remaining, removing it has no
effect on the zero forcing number, by \th\ref{cor:deg_uncritical,prop:uncritical}.  Hence, \(Z(H)=Z(G)-(k_1-1)\), from which the result follows.
\end{proof}

\begin{thm}\th\label{thm:remove_unsaturated_vertices_effect_on_Z}
Let $G$ be a graph that is cobipartite with clique partition $(V_1,V_2)$ such
that the number of unsaturated vertices in $V_1$ and the number in \( V_2 \) are \( k_1 \ge 1
\) and \( k_2 \ge 1 \), respectively. Let $H$ be the graph obtained by deleting all of the
unsaturated vertices in \( G \). Then
\[
Z(G) = Z(H) + (k_1 - 1) + (k_2 - 1).
\]
\end{thm}

\begin{proof}
Let $H'$ be the result of deleting from $G$ all of the unsaturated vertices in $V_1$.  Then the unsaturated vertices of $V_2$ still exist in $H'$, and $H$ is the result of deleting them from $H'$.  One application of \thref{thm:Aunsaturated} gives $Z(G)=Z(H')+(k_1-1)$, while a second application (applied to $H'$) gives $Z(H')=Z(H)+(k_2-1)$.  Combining these equations gives the desired result.
\end{proof}

For many of the results regarding cobipartite graphs in Section \ref{sec:cobipartite}, an essential hypothesis was that the graph be saturated with respect to some partition.
We are now prepared to see that in most cases, this assumption is in fact essential.  For instance, the next result shows that this hypothesis cannot be dropped from \th\ref{thm:triangle number versus ZF for cobipartite}.

\begin{prop}\thlabel{prop:unsaturated_bound_for_Z_one}
Let \( \Y \) be an \( m\times n \) zero-nonzero pattern and let \( G \) be the
cobipartite graph associated with \( \Y \), with \( (V_1,V_2)  \) the corresponding clique
partition of \( G \). If \(
V_1 \) contains \( k_1 \ge 1 \) unsaturated vertices and \( V_2 \) contains no
unsaturated vertices, then   
\[
Z(G) = m + n - \tri(\Y) -1. 
\]
\end{prop}

\begin{proof}
Let $\Y'$ be the pattern that results from deleting from $\Y$ each row and
column consisting entirely of zeros. Then the size of $\Y'$ is $m' \times n'$ where
$m'=m-k_1$. Moreover, $\tri(\Y)=\tri(\Y')$.
Let $G'$ be the result of deleting all unsaturated vertices from $G$.
Then $G'$ is the cobipartite graph associated with $\Y'$ and is saturated with respect to the corresponding clique partition.
Hence, \th\ref{thm:triangle number versus ZF for cobipartite} gives
\begin{equation}\label{eqn:Z_Gprime1}
Z(G') = m' + n - \tri(\Y') \\
    = (m-k_1) + n - \tri(\Y).    
\end{equation}
Using \th\ref{thm:Aunsaturated}, we have
\begin{equation}\label{eqn:Z_Gprime2}
Z(G') = Z(G) - (k_1-1).
\end{equation}
Combining \eqref{eqn:Z_Gprime1} and \eqref{eqn:Z_Gprime2} then gives
\[
Z(G)-(k_1-1) = (m-k_1) + n - \tri(\Y),
\]
which simplifies to
\(Z(G) = m+ n - \tri(\Y) - 1\).
\end{proof}

\begin{prop}\thlabel{prop:unsaturated_bound_for_Z}
Let \( \Y \) be an \( m\times n \) zero-nonzero pattern and let \( G \) be the
cobipartite graph associated with \( \Y \), with \( (V_1,V_2)  \) the corresponding clique
partition of \( G \). If \(
V_1 \) contains \( k_1 \ge 1 \) unsaturated vertices and \( V_2 \) contains \(k_2 \ge 1\)
unsaturated vertices, then   
\[
Z(G) = m + n - \tri(\Y) -2. 
\]
\end{prop}

\begin{proof}
Let $\Y'$ be the pattern that results from deleting from $\Y$ each row and
column consisting entirely of zeros. Then the size of $\Y'$ is $m' \times n'$ where
$m'=m-k_1$ and \( n' = n - k_2 \). Moreover, $\tri(\Y)=\tri(\Y')$.
Let $G'$ be the result of deleting all unsaturated vertices from $G$.
Then $G'$ is the cobipartite graph associated with $\Y'$ and is saturated with respect to the corresponding clique partition.
Hence, \th\ref{thm:triangle number versus ZF for cobipartite} gives
\begin{equation}\label{eqn:Z_Gprime1_two}
Z(G') = m' + n' - \tri(\Y') \\
    = (m-k_1) + (n-k_2) - \tri(\Y).    
\end{equation}
Using \th\ref{thm:remove_unsaturated_vertices_effect_on_Z}, we have
\begin{equation}\label{eqn:Z_Gprime2_two}
Z(G') = Z(G) - (k_1-1) - (k_2-1).
\end{equation}
Combining \eqref{eqn:Z_Gprime1_two} and \eqref{eqn:Z_Gprime2_two} then gives
\[
Z(G)-(k_1-1)-(k_2-1) = (m-k_1) + (n-k_2) - \tri(\Y),
\]
which simplifies to
\(Z(G) = m+ n - \tri(\Y) - 2\).
\end{proof}

The next result now follows immediately from \th\ref{thm:triangle number versus
ZF for cobipartite,prop:unsaturated_bound_for_Z_one,prop:unsaturated_bound_for_Z}.

\begin{thm}\th\label{thm:ZF and triangle number saturated or not}
Let $\Y$ be an $m\times n$ zero-nonzero pattern and let $G$ be
the cobipartite graph associated with $\Y$.
Then $G$ is saturated with respect to the corresponding clique partition if and only if   
\[
Z(G) = m + n - \tri(\Y).
\]
\end{thm}

The above results show how the condition of $G$ being saturated is essential for the relationship between $Z(G)$ and $\tri(\Y)$ given by \thref{thm:triangle number versus ZF for cobipartite}.  Example \ref{ex:bounds not saturated example} showed that this condition is essential also for \thref{thm:not necessarily symmetric min rank result}, which gives the same relationship for the parameters $M(G)$ and $\mr(\Y)$.  The following corollary of \thref{prop:unsaturated_bound_for_Z} allows that example to be generalized.

\begin{cor}\th\label{cor:unsaturated-tri-neq-mr}
Let \( \Y \) be an \( m\times n \) zero-nonzero pattern and let \( G \) be the cobipartite graph associated with \( \Y \). If \( G \) is not saturated with respect to the corresponding clique partition, then
\[
M(G) < m+n-\tri(\Y),
\]
or, equivalently, \( \tri(\Y) < \mr(G) \). \end{cor}

\begin{proof}
When $G$ is not saturated with respect to the clique partition corresponding to \(\Y\), we have
by \th\ref{treewidthdiagram,thm:ZF and triangle number saturated or not} that
\( M(G) \le {Z}(G) < m + n - \tri(\Y) \),
implying that \( \tri(\Y) < \mr(G) \).
\end{proof}

We may now regard Example \ref{ex:bounds not saturated example} as a special case of a more general situation.  In particular, consider any zero-nonzero pattern $\Y$ such that $\mr(\Y)=\tri(\Y)$.  (In particular, by the results of \cite{johnson_and_canto,johnson_and_link,johnson_and_zhang}, any pattern with fewer than $7$ rows will have this property.)  Add at least one row or column of all $0$ entries to $\Y$ to produce $\Y'$.  Then the cobipartite graph $G'$ associated with $\Y'$ will not be saturated with respect to the corresponding partition of $G'$ into cliques, and so \thref{cor:unsaturated-tri-neq-mr} can be applied to give \( \mr(\Y') = \tri(\Y') < \mr(G') \).

Finally, we turn our attention to \thref{thm:bounds_equality_equivalence}, which showed, in the case where $G$ is saturated, that over any infinite field, $M(G)=Z(G)$ is equivalent to $\mr(\Y)=\tri(\Y)$.  It is natural to ask whether this equivalence persists even in the unsaturated case.  In the present work, we were not able to settle this question completely.  The situation is as follows.

In the saturated case, \thref{thm:not necessarily symmetric min rank result} gives
\begin{equation}\label{eqn:M and mr relationship}
    M(G) = m+n-\mr(\Y),
\end{equation}
while \thref{thm:triangle number versus ZF for cobipartite} gives the same relationship for the corresponding combinatorial bounds, namely
\begin{equation}\label{eqn:Z and tri relationship}
    Z(G) = m+n-\tri(\Y).
\end{equation}
By \thref{thm:ZF and triangle number saturated or not}, the introduction of unsaturated vertices creates a gap between the left- and right-hand sides of \eqref{eqn:Z and tri relationship}.
This gap is 
given explicitly by \thref{prop:unsaturated_bound_for_Z_one,prop:unsaturated_bound_for_Z}, and the question is whether an equal gap always occurs between the two sides of \eqref{eqn:M and mr relationship}.
Adding unsaturated vertices to $G$ simply adds rows and columns of all $0$ entries to $\Y$, which clearly has no effect on the right-hand side of \eqref{eqn:M and mr relationship}; the question, then, rests entirely on how adding unsaturated vertices affects the maximum nullity, and whether the effect must match that on the zero forcing number.

\begin{q}\label{q:remove unsaturated effect on M}
Let $G$ be a graph that is cobipartite with clique partition $(V_1,V_2)$ such
that the number of unsaturated vertices in $V_1$ and the number in \( V_2 \) are \( k_1
\) and \( k_2 \), respectively.
Let $H$ be the graph obtained by deleting all
unsaturated vertices from \( G \). Must the equation
\[
M(G) =
	\begin{cases}
		 M(H) + (k_1 - 1) + (k_2 - 1) & \text{if $k_1 \ge 1$ and $k_2 \ge 1$,} \\
		 M(H) + (k_1 - 1) & \text{if $k_1 \ge 1$ and $k_2 = 0$}
	\end{cases}
\]
hold in every case?
\end{q}

In light of \thref{prop:unsaturated_bound_for_Z_one,prop:unsaturated_bound_for_Z}, the equivalence given by \thref{thm:bounds_equality_equivalence} between $M(G)=Z(G)$ and $\mr(\Y)=\tri(\Y)$ persists in the unsaturated case if and only if
Question \ref{q:remove unsaturated effect on M} has an affirmative answer.

\bibliographystyle{plainurl}
\bibliography{refs}

\end{document}